\documentclass{article}
\usepackage[utf8]{inputenc}
\usepackage{amsmath} 
\usepackage{amssymb}
\usepackage{amsthm}
\usepackage{color}
\usepackage{hyperref}
\usepackage{cleveref}
\usepackage{comment}
\usepackage{graphicx}
\usepackage[a4paper]{geometry} 

\usepackage[style=alphabetic]{biblatex} 
\crefname{theorem}{Theorem}{Theorems}
\crefname{thm}{Theorem}{Theorems}
\crefname{thm*}{Theorem*}{Theorems}
\crefname{lemma}{Lemma}{Lemmas}
\crefname{lem}{Lemma}{Lemmas}
\crefname{remark}{Remark}{Remarks}
\crefname{prop}{Proposition}{Propositions}
\crefname{notation}{Notation}{Notations}
\crefname{claim}{Claim}{Claims}
\crefname{defin}{Definition}{Definitions}
\crefname{corollary}{Corollary}{Corollaries}
\crefname{section}{Section}{Sections}
\crefname{figure}{Figure}{Figures}
\crefname{assumption}{Assumption}{Assumptions}

\newtheorem{lem}{Lemma}[section]
\newtheorem{thm}[lem]{Theorem}
\newtheorem{defin}[lem]{Definition}
\newtheorem{assumption}[lem]{Assumption}
\newtheorem{prop}[lem]{Proposition}
\newtheorem{corollary}[lem]{Corollary}

\theoremstyle{definition}
\newtheorem{remark}[lem]{Remark}
\newtheorem{construction}[lem]{Construction}

\renewcommand \P {\mathbb{P}}
\newcommand{\E}{\mathbb{E}}
\newcommand{\C}{\mathbb{C}}

\newcommand{\N}{\mathbb{N}}
\newcommand{\Z}{\mathbb{Z}}
\newcommand{\D}{\mathbb{D}}

\newcommand{\abs}[1]{\lvert #1 \rvert}
\newcommand{\norm}[1]{\lVert #1 \rVert}

\renewcommand{\d}{{\#\delta}}
\newcommand{\cT}{\mathcal{T}}
\newcommand{\ep}{\epsilon}
\newcommand{\Gd}{G^{\# \delta}}
\newcommand{\p}{\partial}

\DeclareMathOperator{\diam}{diam}
\DeclareMathOperator{\fLE}{\overrightarrow{LE}}
\DeclareMathOperator{\bLE}{\overleftarrow{LE}}
\DeclareMathOperator{\harm}{Harm}

\renewcommand{\r}{{\text{ref}}}

\title{Does a portion of dimer configuration determines its domain of definition?}
\author{Antoine Bannier \and Benoît Laslier}

\begin{document}

\maketitle

\abstract{Critical models are, almost by definition, supposed to feature both slow decay of correlations for local observables while retaining some mixing even for macroscopic observables. A strong version of the latter property is that changing boundary conditions cannot have a singular (in the measure theoretic sense) effect on the model away from the boundary, even asymptotically. In this paper we prove that statement for the wired uniform spanning tree and temperleyan dimer model.}

\tableofcontents

\section{Introduction}

The goal of this paper is to establish strong mixing estimates for both the two dimensional wired uniform spanning tree (UST from now on) and the dimer model. Focusing on the former for concreteness and because it will be the main language used in the paper, we wish to say that as soon as we look macroscopically away from the wired boundary, the law of an UST ``forgets'' any small detail of the boundary and only keeps a ``fuzzy memory'' of even the macroscopic location of that boundary. To give a precise sense to the above sentence, we consider the following statistical setup. Fix $U, D_1, D_2$ be three simply connected bounded open sets, with $\overline{U}\subset D_1 \cap D_2$ and let $U^\d, D_1^\d, D_2^\d$ be discretisations of these sets (ultimately we will want to be as general as possible here but, for now, it is enough to think that $U^\d = U \cap \delta \Z^2$ and similarly for the others). Our mixing statement then becomes a testing problem: if one is given the restriction to $U^\d$ of some UST, can we say whether it was sampled from $D_1^\d$ or $D_2^\d$ ? By ``fuzzy memory'', we mean that, even asymptotically as $\delta \to 0$, there is an upper bound on the quality any statistical test can hope to achieve. More precisely our main result is the following.
\begin{thm}\label{thm:main}
	Let $\nu_1, \nu_2$ denote the laws of the restrictions to $U$ of wired UST in $D_1^\d$ and $D_2^\d$ respectively. For all $\epsilon > 0$, there exits $C = C( U, D_1, D_2, \epsilon)> 0$ such that for all $\delta$ small enough
	\[
	\nu_1 ( \frac{1}{C} \leq \frac{d \nu_1}{d \nu_2} \leq C) \geq 1 - \epsilon, \qquad 	\nu_2 ( \frac{1}{C} \leq \frac{d \nu_1}{d \nu_2} \leq C) \geq 1 - \epsilon.
	\]
\end{thm}
Essentially one can think of that result as stating that $\nu_1$ and $\nu_2$ are mutually absolutely continuous in a way which is independent of $\delta$.

In fact \cref{thm:main} and the techniques used to prove it will provide a number of extensions which we hope can form a base set of mixing statements to be used as a toolkit in later work. We defer to \cref{sec:extensions} for any details but briefly \cref{continuity_bounded} contains the aforementioned fact that microscopic details of the boundary are forgotten while \cref{prop:U_multiple} says that the restrictions to several disjoints sets $U$ do not influence each other too much, even if they are nested. The techniques also allow us to study the dimer model with Temperleyan boundary conditions, with essentially the same results as in the spanning tree case.


\smallskip

Before discussing in more details the motivations leading to \cref{thm:main}, let us take a step back and recall briefly some of the history on the UST and dimer model. To the best of our knowledge, the study of UST goes back to the well known matrix-tree theorem of Kirchhoff establishing that on any finite graph, number of spanning trees is given by any cofactor of the Laplacian matrix. This shows that spanning trees are very nice combinatorial objects, non-trivial but still quite amenable to analysis and with a remarkably simple formula appearing in the end. Similarly, the earliest result on the dimer model \cite{MacMahon1915} is a formula for the number a ways to tile an hexagon which is so simple and mysterious the we don't resist the temptation to write it :
\[
\sharp \{\text{lozenge tiling of hexagon with sides $a,b,c$}\} = \prod_{i=1}^a \prod_{j = 1}^b \prod_{k=1}^c \frac{i+j+k-1}{i+j+k-2}.
\]
Strikingly after respectively $160$ and $110$ years of study, there are still both a great number of natural unsolved questions on both models and many different and powerful methods to study them (in part evolved out of the counting arguments above, in part completely different). 

Due to the age and prominence of these models, we will not give a full account of their history but let us still mention a few results on the aspects of the models that will be used later. The bijection between the two models was first described by Temperley \cite{Temperley1974} in the square lattice and then extended to a more general setting in \cite{Kenyon1} and then \cite{Kenyon2007b}. The link between UST and loop-erased random walk (LERW from now on), which serves as the backbone of all the analysis in this paper, was discovered in \cite{Pemantle1991} and refined into an sampling algorithm in the landmark paper of Wilson \cite{Wil2}. We will build extensively on \cite{Schramm2000} which used Wilson's algorithm to derive many qualitative properties of planar UST. Of course in this setting we have to mention the convergence of UST to SLE$_2$/SLE$_8$ from \cite{Lawler2004} (and \cite{Yadin2011} for general planar graphs) which is one of the main motivations for our statement.
Let us also note that in three dimension, LERW is actually one of very few statistical model where a convergence is proved \cite{Kozma2007}. On the dimer side, if we don't assume anything about the boundary the two main known results are a law of large number for the height function \cite{Cohn2001} and a recent local limit \cite{Aggarwal2019}. In this paper however we will only consider so called Temperleyan boundary conditions for which, in a sense, the limit in the law of large number is just the function $0$\footnote{On the other hand, because of the arbitrary choices involved in the defintiion of the height function this is not such a strong condition. In fact for lozenge tiling, \cite{Laslier2021} shows that Temperleyan boundary conditions are in a sense dense in the set of boundary conditions leading to smooth limit shapes.}. In such domain, as mentioned above Kenyon \cite{Kenyon2000,Kdom2} obtained the convergence of fluctuations of the height function in $\Z^2$. This was extended to more general graphs in \cite{BLR16} by exploiting more precisely the link between dimers and spanning trees.

\smallskip

Let us now come back to \cref{thm:main} and discuss why we would like to prove such a statement. First let us point out that for the SLE, the effect of changing the domain of definition on the law of the curve was studied very early (see for example \cite{Lawler2003}) and was a key tool to identify the law of the boundary of Brownian motion. Beyond the special property at $\kappa = \frac{8}{3}$ and $\kappa = 6$, all SLE curves satisfy at least locally \cref{thm:main} and the Radon-Nycodym derivative is in fact explicit. For the Gaussian free field which describe the limit of of the dimer height function, since it is (as its name indicated) a Gaussian process- it is not hard to check with a generalised Girsanov theorem that \cref{thm:main} also holds, again with an explicit Radon-Nycodym derivative. Together with the scaling limit results from the previous paragraph, this proves that a version of \cref{thm:main} holds ``for macroscopic quantities'' but it natural to ask if microscopic details can still carry some information. In fact, we believe that specifically for dimers and UST, keeping track of microscopic information is important because they can be related to several different models (abelian sandpile, double dimer, xor-Ising, $O(1)$ loops on hexagonal lattice) using bijections (or measure preserving maps) that rely heavily on these details. In particular for the double dimer, we plan in a future work to establish Russo-Seymour-Welsh estimates using the results from this paper.

Finally, since for SLE the result is true for all $\kappa$, we believe that analogue of \cref{thm:main} should be true for almost all critical models, and similarly for the extensions of \cref{sec:extensions}. To the best of our knowledge however, even for the critical Ising model the continuity with respect to boundary condition in \cref{continuity_bounded} is not proved and this is a significant limit in our understanding because some of our best techniques are analytic and require smooth boundary conditions.

\medskip

The rest of the paper is organised as follow. \cref{sec:assumptions} completes the fully rigorous statement of \cref{thm:main} by stating the required assumptions on the sequence of graphs $G^\d$ and provides a few properties of the random walk under these assumptions. The rest of \cref{sec:assumptionbackground} contains some background on the objects that will be used later in the paper. The key known results appearing in the proofs are restated for the sake of completeness and ease of reference. More precisely we discuss in that order the loop-soup measure, UST with Wilson's algorithm and the finiteness theorem, some theory of conformal mappings with rough boundary conditions and the link between dimers and spanning tree. \cref{sec:upper_bound,sec:lower_bound} contain the proof of the main theorem. Assuming without loss of generality that $D_1 \subset D_2$ in \cref{thm:main}, \cref{sec:upper_bound} contains the proof of the first part of the statement in \cref{thm:main}  (where in fact only the upper bound is non-trivial), while \cref{sec:lower_bound} proves the second statement. Finally \cref{sec:extensions} treats variants of \cref{thm:main} such as the analogous result for the dimer model or the case where $U$ is an annulus.
Let us emphasize that the reader should be able to skip all or parts of \cref{sec:assumptionbackground} if she is already familiar with the respective topic. The precise assumptions in \cref{sec:assumptions} can be ignored if one always consider $G^\d = \delta \Z^2$ which does \emph{not} make the later arguments significantly easier. Also \cref{sec:domains,sec:winding} are only necessary for the extension to the dimer model.

\paragraph{Acknowledgments :} AB and BL were supported by ANR DIMERs, grant number ANR-18-CE40-0033. We benefited from many discussions with colleagues including Misha Basok, Nathanael Berestycki, Cédric Boutillier, Théo Leblanc and Gourab Ray.

\section{Assumptions and background}\label{sec:assumptionbackground}

In this section, we state more formally our assumptions with respect to the underlying graphs we will be studying (\cref{sec:assumptions}), as well as provide a collection of previous results and definitions needed later in order to make the paper more self contained. A reader familiar with the respective material can certainly skip the sections from \ref{sec:loop_soup} to \ref{sec:winding} and in fact no important idea from this paper would be lost if the reader skips \ref{sec:assumptions} and always take $D^\d$ to be nice subgraphs of $\delta \Z^2$. We finally note that \cref{sec:domains,sec:winding} are only used for the extension of \cref{sec:extensions}.

\subsection{Graphs and random walk}\label{sec:assumptions}

We consider a collection of graphs $(G^{\# \delta})_{\delta \geq 0}$ which are oriented, weighted planar graphs (possibly with self loops and multiples edges). The reader is advised to think of $\delta$ as a scale parameter but note that we are not restricting ourself to scaled version of a fixed graph. Interpreting the weights as transition rates, we obtain a natural continuous time random walk on $G^{\sharp \delta}$ and we denote the law of the walk started from a vertex $v$ by $\P_v$. For an oriented edge $e$, we will denote its weight by $w(e)$ and we denote the set of oriented edges of $G^\d$ by $E_{\delta}$. For ease of notations, we will allow ourself to write an oriented edge from $x$ to $y$ as $x \to y$ as if there was no multiple edges. When using this notation, we will adopt the convention that $w( x \to y) = 0$ if there is no edge from $x$ to $y$. Additionally we make the following assumptions on the collection $G^{\# \delta}$.

\begin{figure}[ht]
	\begin{center}
		\includegraphics[width = .5\textwidth]{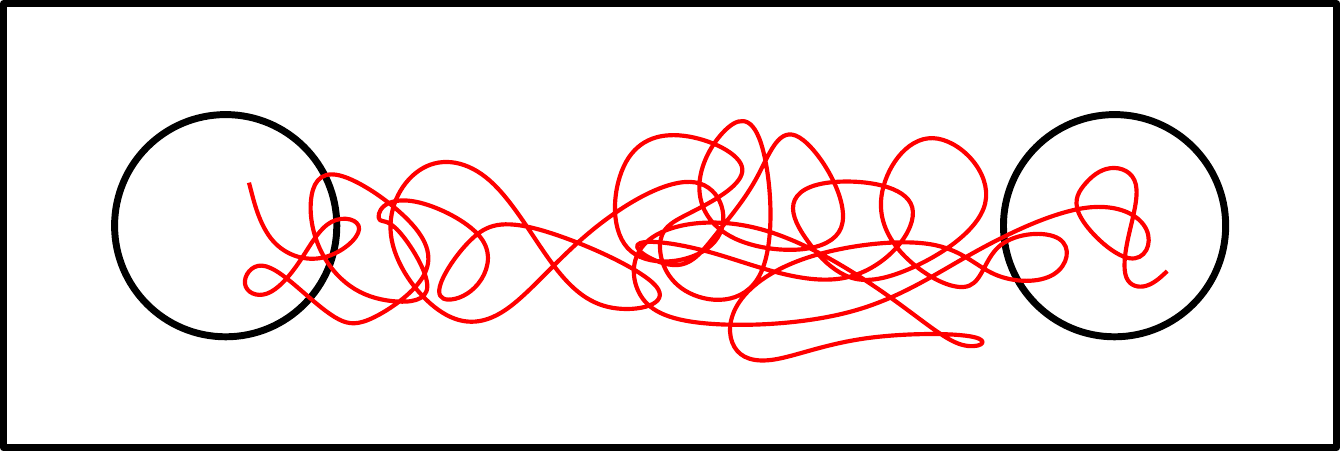}
		\caption{A schematic representation of the crossing estimate.}
		\label{fig:uniform_crossing}
	\end{center}
\end{figure}

\begin{assumption}
\begin{itemize}
	\item \textbf{(Good embedding)} For all $\delta$, the edges of $G^\d$ are embedded in the plane in such a way that they do not cross each other and are piecewise smooth. 
	\item \textbf{(Uniformly bounded density)} There exists a constant $M>0$ such that for all $\delta>0$ and for all $z \in \C$, the number of vertices in the square $z + [0, \delta]^2$ is bounded above by $M$.
	\item \textbf{(Connected)} For any $\delta>0$, the graph $G^\d$ is connected in the sense that it is irreducible for the continuous time random walk: for any two vertices $u$ and $v$, we have $\P_u(RW_1=v) >0$.
	\item \textbf{(Uniform crossing estimate)} We denote by $\mathcal{R}$ (resp. $\mathcal{R}'$) the horizontal (resp. vertical) rectangle $[0,3] \times [0,1]$ (resp. $[0,1] \times [0,3]$). Let $B_1 = B((1/2,1/2), 1/4)$ be the \textit{starting ball} and $B_2 = B((5/2,1/2),1/4)$ be the \textit{target ball}. There exist two universal constants $p>0$ and $\delta_0>0$ such that the following is true. For all $z \in \C$, $\ep >0$, $v \in \ep B_1$ and $\delta \leq \ep \delta_0$ such that $z+v \in \Gd$, we have 
	\begin{equation}
	\P_{v+z}(RW \mbox{ hits } \ep B_2 + z \mbox{ before exiting } \ep \mathcal{R} +z) > p.
	\end{equation}
\end{itemize}
\end{assumption}

Additionally, for the results of \cref{sec:winding} and the statements about the dimer model in \cref{sec:extensions} we will require the following additional assumption
\begin{assumption}
	Edges have uniformly bounded winding. More precisely there exists $C > 0$ such that for any edge $e$ (seen as a curve parameterized with positive speed), for all $s, t$, $|\arg (e'(s) ) - \arg( e'(t) ) | \leq C$, where the argument is taken so that this expression is continuous in $s$ and $t$ away from discontinuity of the derivative and with the natural corresponding convention for jumps.
\end{assumption}

In the following, it will be useful to have a canonical way to associate a subgraph of $G^\d$ to a open set, we therefore introduce the following notation. Given $D$ an open set of the plane, we let $D^\d$ be the subgraph obtained from the vertices and edges which are in $D$ in the embedding of $G^\d$. 
Note to be precise that the convention is that an edge connecting two points inside $D$ but extincting $D$ along the way is \emph{not} included in $D^\d$. This is natural in order to allow slits in $D$. As is usual we will call connected non-empty open sets domains.

When considering the random walk on $D^\d$, we will always use wired boundary conditions in the following sense : we add to $D^\d$ a ``cemetery'' point (later called with a slight abuse of notation $\partial D$) with an outgoing transition rate of $0$. For every edge of $G^\d$ with a starting point in $D^\d$, and touching $\partial D$, we add an edge with the same starting point and weight going to the cemetery. When considering spanning trees of $D^\d$, we will always talk of trees with wired boundary condition and oriented towards the cemetery.

We will also allow ourself to consider a discrete path as the continuous path obtained by concatenating the edges. With a slight abuse of notation, we will also denote by $\gamma [s,t]$ both the subpath of $\gamma$ between times $s$ and $t$ and its support. To be consistent with this convention and our definition of wired boundary condition, we will consider the exit time of a set $D$ to be $\sup \{ t : \gamma[0,t] \subset D\}$, i.e using an edge that starts and ends in $D$ but intersects $\partial D$ counts as an exit. To be consistent with the continuous time random walk, we will sometime write $X_{T} = e$ for an edge $e$ if $T$ is a time where a jump along the edge $e$ occurs.

\medskip

In the following, we will only be interested by (compact) curves up to their time parametrisation (in fact note that our assumptions on the random walk are invariant under time change). When considering the distance between curves, we will always use the $L^\infty$ distance up to time reparametrisation, i.e
\[
d( \gamma_1, \gamma_2) = \inf_{f, g} \sup_{t} | \gamma_1(f(t)) - \gamma_2(g(t)) |,
\]
where $f$ and $g$ are continuous, strictly increasing, from say $[0, 1]$ onto the domain of definition of $\gamma_1$ and $\gamma_2$. 
We will also say that a sequence of rectangles (which are translated and scaled versions of $\mathcal{R}$ and $\mathcal{R}'$) is $\epsilon$-close to a curve $\gamma$ if none of the long sides of the rectangles is larger than $\epsilon/2$, all the starting and target balls match in order and the curve connecting the center of these balls by straight lines in order is at distance at most $\epsilon/2$ of $\gamma$. 

We now state a few consequences of our assumptions on the walk. First a simple statement saying that the random walk has a positive probability to follow any trajectory.
\begin{lem}\label{lem:RWfollow}
	For any compact curve $\gamma$ from $[0,1]$ to $\C$, for any $\epsilon >0$, there exists $p > 0$ and a stopping time $T$ such that for all $\delta$ small enough, for all $v \in B( \gamma[0], \epsilon)$, we have
	\[
	\P_v( d(X[0, T], \gamma ) \leq  \epsilon ) > p.
	\]
\end{lem}
\begin{proof}
	Note that with the uniform topology on curves up to re-para\-metri\-sa\-tion, any $\gamma$ curve can be approximated by a finite concatenation of straight lines up to precision $\epsilon$. Indeed define a sequence of times $T_i = \inf\{ t > T_{i-1} : \gamma( t ) \not \in B( \gamma( T_{i-1}, \epsilon/2))\}$. By uniform continuity of $\gamma$, this sequence must contain finitely many terms and by construction the distance between $\gamma$ and the piecewise linear curve connecting the $\gamma( T_i)$ in order is at most $\epsilon$.
	
	Clearly any piecewise linear curve can be approximated by a finite sequence of rectangles and, by uniform crossing and the Markov property, the random walk has a positive probability to cross all of them in order. This concludes the proof.
\end{proof}

By considering the possibility that a walk does a ``full turn'' in an annulus, we deduce a Beurling estimate (Lemma 4.17 in \cite{BLR16}).
\begin{lem}[Non quantitative Beurling estimate]\label{beurling}
	There exists $\alpha > 0$ such that for all $r, R$ with $R \geq 2r$, for all $\delta$ small enough, for all $v \in G^\d$ and for any set $A$ connecting $\partial B( v, r)$ to $\partial B( v, R)$, 
	\[
	\P_v( RW \text{ exits } B( v, R) \text{ before hitting } A) < (\frac{r}{R})^\alpha .
	\]
\end{lem}
We will also need a dual statement about the probability that the random walk first exits a domain from a marked point.
\begin{lem}[Harmonic measure bound]\label{lem:harmonic_measure}
	There exists $\alpha$ such that for all $r, R$ with $R \geq 2r$, for all $\delta$ small enough, for all simply connected open set $D$, and all points $v, w$ with $v \in D^\d$ and $w \in \partial D$ and $\abs{v-w} > R$, we have
	\[
	\P_v(\text{ RW exits $D$ inside $B(w, r)$}) < ( \frac{r}{R})^\alpha .
	\]
\end{lem}
The proof of this estimate is identical to the Beurling estimate with just $w$ taking the role of infinity so we leave details to the reader.

Arguments similar to the Beurling estimate prove that discrete harmonic functions must satisfy the Harnack inequality. Together with the representation of a walk conditioned on its endpoint as an h-transform, this implies that the uniform crossing also holds for walks conditioned on their exit point from a domain, at least as long as the rectangle that must be crossed stays far from the boundary (see Lemma 4.4 in \cite{BLR16} which is stated for annuli but only uses the Harnack inequality).
\begin{lem}[Conditional uniform crossing]\label{crossing_conditional}
	Let $\mathcal{R}$, $B_1$, $B_2$ be as in the uniform crossing assumption. For all $ \alpha >0$ there exists $p>0$ and $\delta_0 >0$ such that the following holds. For all $z\in \C$, $\epsilon >0$, $\delta \leq \epsilon \delta_0$, $v \in (\epsilon B_1 + z)^\d$, $D$ such that $d( \partial \epsilon \mathcal R + z, \partial D) > \alpha \epsilon$ and en edge $w$ such that $\P_v( X(T_D) = w ) >0$,
	\[
	\P_{v+z}(X \mbox{ hits } \epsilon B_2 + z \mbox{ before exiting } \epsilon \mathcal{R} +z  |  X(T_D) = w  ) > p,
	\]
		where $T_D$ is the exit time from $D$.
\end{lem}

Finally we will need an estimate saying that if a random walk is conditioned to exit at a certain point and starts close to it, then it is unlikely to ever go far away. If $D$ is an open set, denote by $T_D$ the exit time from that set (in the sense mentioned in the beginning of the section).
\begin{lem}
	There exists $C$ and $p >0$ such that for any open set $D$ with a smooth boundary, for all $r >0$ small enough, and all $\delta$ small enough the following holds. For all $v$ with $v \in D^\d$ and all edges $w$ such that $\P(X_{T_D} = w) > 0$ and $2r \geq d( v, \partial D) \geq r$, we have
	\[
	\P_v( X[0, T_D] \subset B(w, Cr) | X(T_D) = w ) > p.
	\]
	Furthermore $p$ can be chosen close to $1$ by taking $C$ large enough.
\end{lem}
\begin{proof}
	Fix $D, v, w$ and $r$ as in the statement of the lemma, where the condition on the size of $r$ will be made precise later. Let us also write $h(u) := \P_u (X(T_D) = w )$ and note that this is harmonic for the random walk on $D^\d$. By the Markov property and Beurling's estimate
	\begin{align*}
	\P_v(  X[0, T_D] \not\subset & B(w, Cr), X(T_D) = w )  \\
	&\leq \P_v(  X[0, T_D] \not\subset B(w, Cr) ). \max_{\partial B(x, Cr)^\d \cap D^\d} h(u) \\
	& \leq \frac{1}{(C-1)^\alpha} \max_{ \partial B(x, Cr)^\d \cap D^\d} h(u),
	\end{align*}
where we use a maximum since $\partial B(x, Cr)^\d \cap D^\d$ is finite. Let $\gamma$ be a path in $ B(w, Cr) \cap D$ connecting a point where this maximum is reached to $w$ and along which $h$ is increasing. Note that by harmonicity of $h$,
\[
\P_v(  X[0, T_D] \not\subset B(w, Cr), X(T_D) = w ) \leq \frac{1}{(C-1)^\alpha} \frac{h(v)}{\P_v(X_{T_D \wedge T_\gamma} \in \gamma )},
\]
where $T_\gamma$ is the hitting time of $\Gamma$. Note that $\gamma$ separates $B( w, Cr) \cap D$ into exactly two connected components, call $A_1$ the component containing $u$ and $A_2$ the other one. Since $D$ is smooth, if $r$ is chosen small enough $\partial D$ is almost linear in a neighbourhood of $w$ and $\partial A_2 \cap \partial D$ contains an arc connecting $w$ and a point on $\partial B(w, Cr)$. By \cref{lem:RWfollow}, the random walk starting in $v$ has a positive probability $p>0$ (independent of $C$) to follow $\partial B( w, |v-w|)$ in the direction leading to $\partial A_2$ up to an error $r/2$ and on that event we must have $T_\gamma < T_D$. Overall,
\[
\P_v(  X[0, T_D] \not\subset B(w, Cr), X(T_D) = w ) \leq \frac{1}{(C-1)^\alpha} \frac{h(v)}{p},
\]
which concludes.
\end{proof}

\subsection{Loop soup}\label{sec:loop_soup}

In this section, we recall the definition of the random walk loop soup measure as well as its relation to the loop-erased random walk. We also provide a basic estimate on the mass of macroscopic loops in the measure. We will keep the exposition short and refer to chapter 9 in \cite{Lawler2010} for much more details on this topic.

Let us denote by $q( v \to v')$ the discrete transition rate from $v$ to $v'$, i.e the edge weights normalised such that $\sum_{v'} q( v \to v') = 1$. We call a finite sequence $\gamma =(v_0, \ldots, v_n)$ such that $v_0 = v_n$ a \emph{rooted loop} and call $v_0$ its $\emph{root}$. We call length of the loop simply its number of steps, i.e $\abs{(v_0, \ldots, v_n)} = n$. Note that there is no restriction to simple loops and that the length of a loop is can be much smaller than the size of its support if one has many multiple point. An \emph{unrooted loop} (which we also just call a loop later) is an equivalence class of rooted loop by cyclical permutation.
\begin{defin}
	The \emph{rooted loop measure} $\Lambda = \Lambda^\d$ on $G^\d$ is the measure on rooted loops defined as
	\[
	\Lambda (v_0, \ldots, v_n) = \frac{\prod_{i=0}^{n-1} q( v_i \to v_{i+1})}{n}.
	\]
	The unrooted loop measure is the image of the rooted measure by the natural ``forgetting the root'' map\footnote{Typically the mass of an unrooted loop is the product of transition probabilities but there are small complications if a loop goes  several times around the same path.}. We call \emph{loop soup} of $G^\d$ the Poisson process with intensity given by the unrooted loop measure.
	
	The loop measure on $D^\d$ for a subgraph $D^\d$ is simply the restriction of $\Lambda$ to loops that do not exit $D^\d$ and the loop soup on $D^\d$ is similarly the restriction of the loop soup.
\end{defin}

In this paper, the loop soup will be used to describe the conditional law of a random walk given its loop-erasure.

\begin{construction}\label{cons:walk_given_erasure}
	Fix a subgraph $D^\d$, a path $\gamma = (\gamma_1, \ldots, \gamma_n)$ from some $v \in D^\d$ to $\partial D^\d$ and a finite collection of loops $(\ell_k)_{k\geq 0}$. We will define a path $X$ as follow.
\begin{itemize}
	\item Apply a uniform permutation of the indices to the finite sequence $(\ell_k)$ (we however keep the same notation for simplicity).
	\item Define $K_1 = \{ k : \gamma_1 \in \ell_k \}$, for each $k \in K_1$ let $\tilde \ell_k$ denote a version of $\ell_k$ rooted at $\gamma_1$ (choose uniformly over all possibilities if necessary) and define $L_1 = \bigoplus_{k \in K_1} \tilde \ell_k$, i.e $L_1$ is the concatenation of all the loops that intersect $\gamma_1$. Note that when concatenating rooted loops, one must omit some repetition of the root but we hope the definition is clear.
	\item Define $K_2 = \{ k : k \not \in K_1 , \gamma_2 \in \ell_k \}$ and set $L_2 = \bigoplus_{k \in K_2} \tilde \ell_k$ with similar notation as above.
	\item Define  $K_3 = \{ k : k \not \in K_1 \cup K_2 , \gamma_3 \in \ell_k \}$ and so on up to $K_{n_1}$ and $L_{n-1}$.
	\item Finally we set $X = \big( \bigoplus_{i=1}^{n-1} L_i \big)\oplus \gamma_{n}$. Note that the edges of $\gamma$ appear as the transitions from $L_i$ to $L_{i+1}$. 
\end{itemize}
\end{construction}

The above procedure is motivated by the following result which can be found for example in \cite{Jan2011} as Proposition 28 together with Remark 21.
\begin{lem}\label{walk_given_erasure}
	Fix $D^\d$ together with a path $\gamma$ ending in $\partial D^\d$ and apply the construction \ref{cons:walk_given_erasure} to a sample $(\ell_k)$ of the loop soup on $D^\d$. The law of the resulting $X$ is given by the conditional law of a random walk stopped when it exits $D^\d$ given that it's loop-erasure if $\gamma$.
\end{lem}

\begin{remark}
	There is a converse statement which says that if one sample a wired uniform spanning tree of $D^\d$ using Wilson's algorithm (see next section), then the set of erased loops forms in a sense an instance of the loop soup of $D^\d$. There are however a few subtle points in such a statement, mainly in how to go from the loops $L_i$ to the $\ell_k$ in the previous notations. We do not go into more details here since we will not need the converse statement.
\end{remark}

The last ingredient needed is a bound on the mass of large loops. For general planar graphs satisfying uniform crossing, this was obtained in Corollary 4.22 from \cite{BLRtorus2}.
\begin{lem}\label{loop_mass}
	For all $\epsilon > 0$, there exits $C>0$ and $\delta_0$ such that for all $r > 0$, for all $\delta \leq r \delta_0$, for all $z \in \C$ the following holds. Let $\mathcal{L}(\epsilon, r, z, \delta)$ be the set of loops staying in $B( z, r)^\d$ and with diameter greater that $\epsilon r$, we have
	\[
	\Lambda( \mathcal{L}(\epsilon, r, z, \delta)) \leq C.
	\]
\end{lem}

\subsection{Uniform spanning tree}\label{sec:UST}

We now turn to the main object from this paper, and while we assume that the reader should be familiar with it, we recall briefly the main properties that will be needed in the rest of the paper.

In an oriented setting as ours, a spanning tree of a finite graph $G$ rooted at some vertex $\partial$ is a subset of oriented edges $\cT$ which contains no cycle and such that, in $\cT$, $\partial$ has no outgoing edge and every other vertex has a single outgoing edge in $\cT$. It is easy to see that from any vertex $v \neq \partial$, following the outgoing edges in order defines a simple path from $v$ to $\partial$ which we naturally call the branch of $\cT$ starting in $v$. This branch will in general be denoted $\gamma_v$. In this paper we will always consider graphs of the form $D^\d$ with wired boundary condition and root the tree at the wired vertex.
The uniform spanning tree (UST) is the measure such that
\[
\P( \cT ) \propto \prod_{(vv') \in \cT} q( v \to v').
\]

\begin{remark}
	In the most classical case of an un-oriented graph (or equivalently a reversible Markov chain), forgetting the orientation gives a law on un-oriented trees which is independent of the choice of root. This does not hold with general oriented graphs however.
\end{remark}

It is well known that the UST can be sampled through Wilson's algorithm which we recall rapidly for completeness. Start with an initial trivial tree $\cT_0 = \{\partial\}$ and proceed by induction, constructing subtrees $\cT_i$ as follows. Given $\cT_i$, pick any vertex (not in $\cT_i$ otherwise the step is trivial) and starts a random walk from it until it hits a vertex of $\cT_i$. Erase the loops of this walk (say in forward time) to obtain a simple path and set $\cT_{i+1}$ as the union of this path and $\cT_{i}$.

One of our key tools to understand this sampling procedure and the geometry of the UST in general is Schramm's finiteness theorem \cite{Schramm2000} which intuitively says that after a large but finite (and independent of the underlying mesh size) number of steps in the algorithm, all the macroscopic structure of the tree is already determined and only small local details still need to be sampled. Here we use the following version of that theorem which goes back to \cite{BLR16}, Lemma 4.18.
\begin{lem}\label{thm:finitude}
	Fix $D$ a simply connected bounded domain and $U$ open such that $\bar U \subset D$. For all $\epsilon >0$, there exits $k \in \N$ such that for all $\delta$ small enough, one can find $k$ vertices of $D^\d$ within distance $\epsilon$ of $U$ such that the following holds. Suppose that we start by sampling the branches from $v_1$ up to $v_k$ in Wilson's algorithm and let $\cT_k$ be the corresponding subtree. There is a way to choose points $v_{k+1}, \ldots$ where we still only consider points in within distance $\epsilon$ of $U$ and eventually exhaust all the vertices of $U$ (call the resulting subtree $\cT_U$) so that, except on a event of probability at most $\epsilon$,
	\begin{itemize}
	\item No connected component of $\cT_U \setminus \cT_k$ has diameter more than $\epsilon$,
	\item for all $i \geq k+1$, the random walk started from $v_i$ hits $\cT_{i-1}$ before exiting $B(v_i, \epsilon)$.
\end{itemize}
\end{lem}
Essentially, this statement means that in order to sample all the branches starting in $U$ of an UST, it is enough to understand well the behaviour of the first $k$ branches (which unfortunately have to start in a neighbourhood of $U$) and everything else will be made of very small subtrees attached to these branches first few branches. Furthermore and crucially for our purpose, these small subtrees are generated by erasing walks that never went far from $U$ and are therefore likely easy to couple between different domains.

\subsection{Uniformisation and prime ends}\label{sec:domains}

In the future, we want our results to be usable out of the box say for the law of a spanning tree conditioned on one of its branch. By Wilson's algorithm, this just means that we want to be able to apply our results on domains slitted by a rough curve, including discussing winding of curves in such domain for the application to dimers.
For that purpose, we now introduce a few elements on uniformisation of domains with rough boundaries. Unless specified otherwise, all results are found in \cite{Pommerenke1992} chapter 2 and they will only be used in \cref{sec:extensions}.

\begin{defin}
	A closed set $A \subset \C$ is locally connected if for all $\epsilon$, there exists $\eta$ such that for all $x, y \in A$, if $|x-y| \leq \eta$ then there exists a connected subset $B$ of $A$ such that $x, y \in B$ and $\diam(B) \le \epsilon$.
\end{defin}

\begin{defin}
	Let $D$ be a domain, the diameter distance $d_D$ on $D$ is defined as
	\[
	d_D( x, y) = \inf_{\text{$\gamma$ connecting $x$ to $y$}} \operatorname{diam}(\gamma),
	\]
	where in the above infimum $\gamma$ must be a continuous path in $D$.
	the function $d_D$ is a distance on $D$ and we let $\overline{D}_d$ denote the completion of $D$ for this distance and $\partial_d D = \overline{D}_d \setminus D$.
\end{defin}
As an aside, note that one could also define a natural intrinsic distance using the length of curves instead of their diameter. A moment of though shows however that the length distance can be much finer than the diameter one close to a rough boundary and for our purpose we really need to consider the diameter.

\begin{thm}\label{thm:pomerenke}
	Let $D$ be a bounded simply connected domain, and let $f$ denote any conformal map mapping $\D$ to $D$. The following five conditions are equivalent
	\begin{itemize}
		\item $f$ has a continuous extension to $\overline{\D}$,
		\item $f$ extends to an homeomorphism from $\overline{\D}$ to $\overline{D}_d$.
		\item $\partial D$ can be parametrised as a continuous loop, i.e there is a continuous $\gamma$ from the unit circle to $\C$ such that $\partial G = \gamma( \mathbb{S}_1)$,
		\item $\partial D$ is locally connected,
		\item $\C \setminus D$ is locally connected. 
	\end{itemize}
\end{thm}
In particular we can always assume in the third line that the parametrisation of the boundary is of the form $\gamma( e^{i \theta}) = f( e^{i \theta})$ where $f$ is conformal. The second statement in this theorem is not given in \cite{Pommerenke1992} but can be found in \cite{Herron2012}.

In this context, we will call an element of $\partial_d D$ a prime end. This is a slight abuse of terminology since the true notion of prime end is more general, we refer the interested reader to \cite{Pommerenke1992} for the actual definition and the fact that it matches our usage in the locally connected case. 

Clearly the diameter distance on $D$ is finer that the restriction to $D$ of the usual euclidean distance and therefore there is a canonical onto map from $\partial_d D$ to $\partial D$, i.e each prime end is associated to a single boundary point but several prime end can share the same image (for examples the two sides of a slit are different prime ends). It is also easy to see using uniform continuity that if $\gamma$ is a continuous path from $[0, 1]$ to $\C$ such that $\gamma[0, 1) \subset D$ then $\gamma$ can be extended to a continuous path from $[0,1]$ to $\overline{D}_d$. In particular any such curve ending on $\partial D$ determines a unique prime end which will be important later in the paper.

Finally, let us introduce a notion of distance between domains with a marked prime end that will be useful in \cref{sec:extensions}.
\begin{defin}\label{def:distance_domain}
	Let $A$, $B$ be two bounded simply connected domains with locally connected boundaries and let $x \in \partial_d A$ and $x \in \partial_d B$ be two prime ends. For $\epsilon>0$, we say that the marked domains $(A, x)$ and $(B, y)$ are $\epsilon$-close if there exists an homomorphism $f : \overline{A}_d \to \overline{B}_d$ such that $f(x) = y$, for all $a \in A$, $|f(a)-a| \leq \epsilon$ and for all $b \in B, |f^{-1}(b) -b| \leq \epsilon$.
\end{defin}
Let us note that we will not actually need to show that the above definition defines a distance. Thanks to \cref{thm:pomerenke}, we have a simple (and trivial) criterion in terms on uniformisation maps for two marked domains to be close.
\begin{lem}\label{lem:uniformisation_criterion}
	Let $A, B, x, y$ be as above and let $f : \D \to A$ and $g : \D \to B$ be two uniformisation maps such that $f(1) = x$ and $g(1) = y$. If $\norm{f-g}_{L^\infty ( \overline \D)} \leq \epsilon$ then $(A, x)$ and $(B, y)$ are $\epsilon$-close. 
\end{lem}

As shown in \cref{fig:slitted_domains}, the notion of distance from \cref{def:distance_domain} cannot be simply rewritten in terms of Caratheodory's topology or in terms of the distance between the boundary seen as curve so we unfortunately cannot provide a simple direct geometric description of it. However we can still obtain the following non-uniform continuity statement
\begin{figure}[ht]
	\begin{center}
		\includegraphics[width=.7\textwidth]{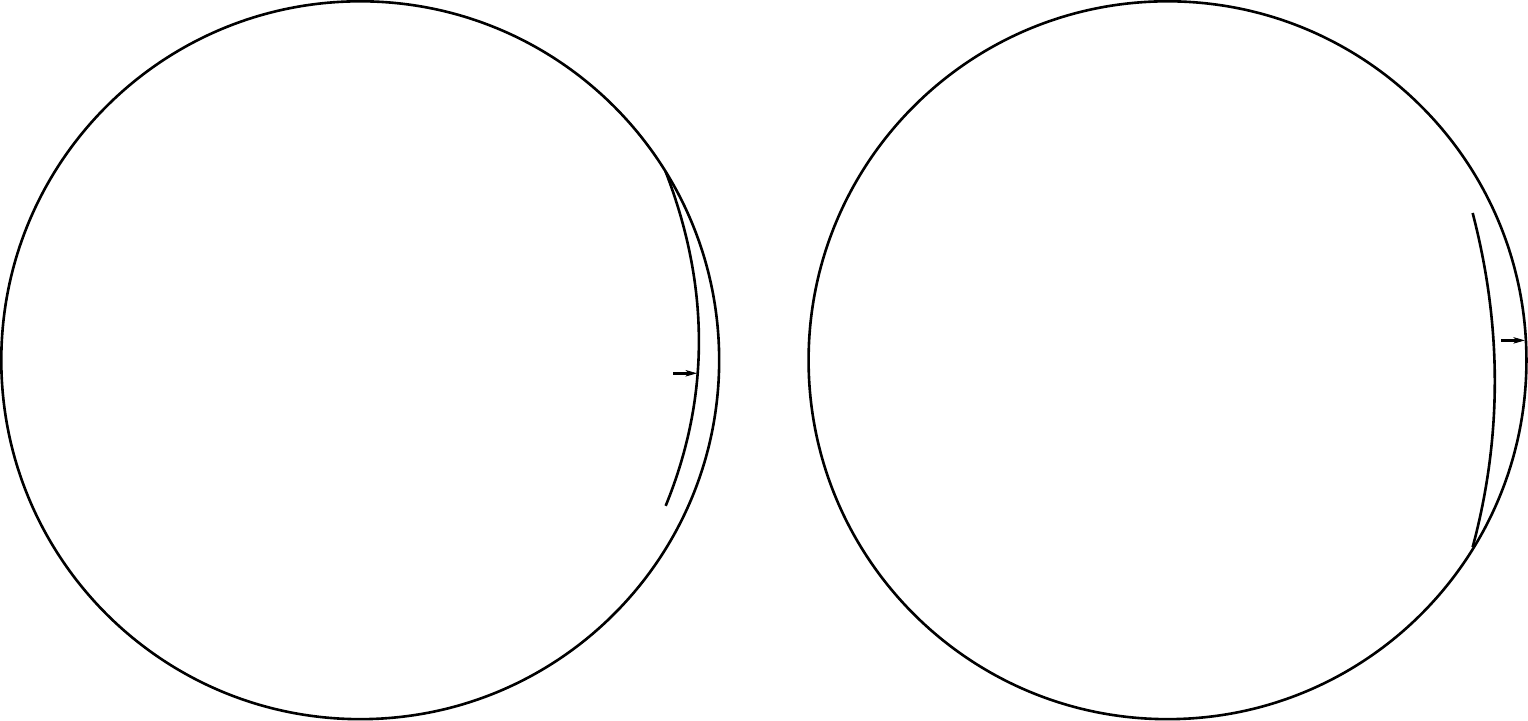}
		\caption{Two slitted versions of the unit disc with a marked prime end denoted by an arrow. Note that both domains are close to the regular unit disc for Caratheodory's topology with respect to the center of the disc. Also, when seen as curve parametrised from the marked prime end, the boundary of these domains are close. Yet the two marked domains are \emph{not} close in the sense of \cref{def:distance_domain}}\label{fig:slitted_domains}
	\end{center}
\end{figure}
\begin{lem}\label{slit_continuity}
	Fix $A$ a bounded simply connected domain with locally connected boundary and $x \in \partial_d A$ a prime end. For every $\epsilon >0$, there exists $\eta > 0$ such that for all bounded simply connected domain with locally connected boundary $B$ and prime end $y \in \partial_d B$ the following holds. Let $\gamma_A$ and $\gamma_B$ be two parametrisation of $\partial A$ and $\partial B$ as close curves starting from $x$ and $y$ respectively. If the uniform distance up to reparametrisation between $\gamma_A$ and $\gamma_B$ is smaller than $\eta$ then $(A, x)$ and $B(y)$ are $\epsilon$-close.
\end{lem}
In other word, every pair $(A, x)$ is a continuity point for the map between the distance on boundary curves to the one from \cref{def:distance_domain}. Let us also note that the arbitrary nature of the choice of parametrisation for $\gamma_A$ and $\gamma_B$ above is irrelevant with our choice of distance.
\begin{proof}
	Fix $A, B, x, y, \gamma_A, \gamma_B$ as above together with some $z \in A$. Since as noted above the choice of $\gamma_A$ is irrelevant, we can assume without loss of generality that $\gamma_A : [0, 2\pi] \to \partial A$ is obtained from the map conformal map $f : \D \to A$ sending $0$ to $z$ and $1$ to $x$ and that $\gamma_B$ is also defined on $[0, 2\pi]$ with $\norm{\gamma_A - \gamma_B}_\infty \leq \eta$ for some $\eta$ to be taken small enough later. In fact let us already assume that $\eta < d( z, \partial A)$ and note that this imply that $z \in B$. Finally let $\harm_A (\gamma_A[0, s])$ denote the harmonic measure in $A$ seen from $z$ of $\gamma_A[0, s]$ seen as a set of prime end, with a similar notation for $B$.
	
	Consider a planar Brownian motion $B_t$ started from $z$ and let $\tau_A$ (resp. $\tau_B$) be its exit time from $A$ (resp. $B$). Also let $T = \inf \{ t > 0 : d( B_t, \partial A) \leq 2 \eta\}$. Note that by Beurling's estimate (\cref{beurling}) there exits $\alpha > 0$ (independent of $A$, $B$) such that $\P( d_A(B_T - B_{\tau_A}) \geq \sqrt{\eta}) \leq \eta^\alpha$ and similarly for $B$. In particular for all $s > 0$, we have
	\[
	\P( d_A ( X_T , \gamma_A[0, s]) \geq \sqrt{\eta}) \leq \P( X_\tau \not \in \gamma_A[0, s]) + \eta^\alpha \leq 1 - s + \eta^\alpha.
	\]
	On the other hand, by \cref{thm:pomerenke} $\gamma^{-1}_A$ seen as a curve from $\overline{A}_d$ to $\overline{B}$ is continuous but therefore also uniformly continuous. In particular we can find $\beta > 0$ (depending only on $A$ and going to $0$ with $\eta$) such that for all $s, s'$ with $|s-s'| \geq \beta$, we have $d_A(\gamma_A(s), \gamma_A(s') )  \geq 2\sqrt{\eta}$. Again using Beurling's estimate, we obtain that for all $s \geq 2 \beta$
	\[
	\P( d_A(X_T , \gamma_A[\beta, s-\beta])  \leq \sqrt{\eta} ) \leq  \P( X_\tau \in \gamma_A[0, s]) + \eta^\alpha \leq s - \eta^\alpha.
	\]
	Combining the two above estimates, we see that for all $s \geq 2 \beta$,
	\[
	s - \eta^\alpha \leq \P( d_A ( X_T , \gamma_A[0, s]) \leq \sqrt{\eta}) \leq s + 2 \beta + \eta^\alpha.
	\]
	
	Let us now turn to estimates related to $B$. Suppose that $d_A ( X_T, \gamma_A[\beta, s-\beta] \leq \sqrt{\eta}$, then by definition of $\beta$, $d_{eucl}( X_T, \partial A \setminus \gamma_A[0, s] ) \geq \sqrt{\eta}$ but then we also have $d_{eucl}( X_T, \partial B \setminus \gamma_B[0, s] ) \geq \sqrt{\eta} - \eta$ and by Beurling
	\[
	\P( X_{\tau_B} \in \gamma_B[0, s] | d_A ( X_T, \gamma_A[\beta, s-\beta] \leq \sqrt{\eta} ) \geq 1 - 2 \eta^\alpha,
	\]
	and therefore for all $s$,
	\[
	\harm_B (\gamma_B[0, s]) \geq (1 - 2 \eta^\alpha) (s - 2\beta - \eta^\alpha).
	\]
	Since the bounds are arbitrary, we have a corresponding upper bound by considering the complementary event.
	
	To conclude we will apply \cref{lem:uniformisation_criterion}. To that end, let $\tilde \gamma_B$ be the reparametrisation of $\gamma_B$ by harmonic measure, i.e such that $\harm_B (\tilde \gamma_B[0, s]) = s$ for all $s$. Note that this implies that $\tilde \gamma_B$ is the boundary extension of the conformal map $g : \D \to B$ sending $0$ to $z$ and $1$ to $y$.
	 On the other hand, combining the uniform continuity of $\gamma_A$ and the previous bound, it is easy to see that we can find $\epsilon$ independent of $B$ and going to $0$ with $\eta$ such that $\norm{\gamma_A - \tilde \gamma_B} \leq \epsilon$. By the maximum principle, this shows that $|f - g| \leq \epsilon$ on $\overline{\D}$ which concludes by \cref{lem:uniformisation_criterion}.
\end{proof}

\subsection{Dimers, winding, and height function}\label{sec:winding}

In this section we recall some simple facts about winding of open curves and then present the connection between uniform spanning tree and dimer model. We use the terminology from \cite{BLRtorus2}.

\begin{defin}
	Let $\gamma : [0, 1] \to \C$ be a continuous curve. For $z \not \in \gamma( [0, 1])$, we define the (topological) winding of $\gamma$ around $z$ to be 
	\[
	W( \gamma, z) = \arg( \gamma(1) - z) - \arg( \gamma(0) - z),
	\]
	where the argument is taken so that it is continuous along $\gamma$. If $\gamma$ admits a tangent at its starting, we set $	W( \gamma, \gamma(0)) = \arg( \gamma(1) - \gamma(0)) - \lim_{\epsilon \to 0}\arg( \gamma(\epsilon) - \gamma(0))$ and similarly for the ending point.
\end{defin}

For smooth curve, there is also a natural ```intrinsic'' definition of winding that does not need a reference point since one can track the argument of the derivative. It is related to our definition though the following proposition:
\begin{prop}
	Let $\gamma : [0, 1] \to \C$ be a smooth injective curve such that $\gamma'(t) \neq 0$ for all $t \in [0, 1]$. We have
	\[
	W( \gamma, \gamma(0)) + W( \gamma, \gamma(1)) = \arg( \gamma'(1)) - \arg( \gamma'(0)),
	\]
	where the arguments on the right hand side are continuous along $\gamma'$.
\end{prop}
Thanks to this proposition, we can generalise the intrinsic notion of winding to any simple curve which is regular close to its endpoints even if it is globally very rough.  Also since $W( \gamma, z)$ is a topological quantity, it is not hard to see that the formula also holds for limits of simple curves, i.e we can think of $W( \gamma, \gamma(0)) + W( \gamma, \gamma(1))$ as an intrinsic winding for any non self-crossing curve.

Let us now describe the first map from a spanning tree to a dimer configuration in the context of the discretization procedure of \cref{sec:domains}, see \cite{Kenyon1} for the original description of that map and \cite{BLRtorus2} for a detailed analysis of the height function in terms of winding.

We start with a simply connected bounded open set $D$. Fix a \emph{reference prime end} of $D$, associated to a point $x_{\r}$ of $\partial D$ and a simple path $\gamma_{\r}$ starting from $x_{\r}$ and going to infinity. For simplicity, let us assume that $\gamma_{\r}$ only intersects $\overline{D}$ at $x_{\r}$ (in particular $x_{\r}$ this imply that $x_{\r}$ was not chosen on a slit). The interested reader can check that we could drop this assumption by interpreting $\gamma_{\r}$ correctly as a path intersecting but not crossing $\partial D$. Also let us assume that $\gamma_{\r}$ merges with the positive real axis at some point and goes to infinity along the positive real axis. It will be clear from the construction later that we do not lose generality assuming this and it will be convenient when comparing different domains later.

Recall that we see $D^\d$ as a wired graph where every edge of $G^\d$ intersecting $\partial D$ was replaced by an edge directed to a single boundary vertex. We now give a slightly different interpretation of the same graph as follows. For each edge (which remember is always oriented) intersecting $\partial D$, we look at the path from the starting point up to the first intersection with $\partial D$. As noted after \cref{thm:pomerenke} we can see these paths as ending at prime ends so we add a single vertex for each of those prime end.

 Since $\partial D$ is locally connected, we can see it as a continuous curve starting and ending at the reference prime end. This produces an ordering of all the extra vertices added in the previous step starting just after the marked prime end and we add an oriented edge from each prime end to the next one, following $\partial D$, adding if necessary an extra vertex at the marked prime end. By design, this construction is equivalent in terms of random walk to the original one : we have simply replaced the single boundary point by a path following the boundary with only deterministic transitions until it reaches $x_\r$.

Now consider the planar dual of $D^\d$ and note that according to the above construction it is natural to have a dual vertex for each face touching the boundary. We embed the dual such that each dual vertex is in the corresponding face, dual pairs of edges intersect exactly once and other pairs of edges do not intersect away from their endpoints. The superposition graph of $G^\d$ and its dual is defined as the graph whose vertex set is the union of all primal vertices, dual vertices, and intersection points between primal and dual edges and whose edge set is the set of ``half edges'' connecting a primal/dual vertex to an intersection vertex. We see it as a weighted graph by giving weight 1 to all half dual edges and weight $p(x \to y)$ to the half edge from a primal vertex $x$ to the intersection along the edge $x \to 1$. We also call reference dual vertex the vertex adjacent to the part of the boundary just after the reference prime end.

\begin{thm}\label{thm:bijection}
	There is a measure preserving bijection between the uniform spanning tree measure on $D^\d$ and the dimer measure on the superposition graph with the boundary primal vertices and the reference dual vertex removed.
\end{thm}

There is also a way to read the height function from the tree in this bijection. For this we need to introduce the winding field of a tree. For each face of the superposition graph, choose a smooth path connecting the adjacent primal and dual vertices and choose a marked point inside this face. We call these paths diagonals. For each face $f$, we call $\gamma_f$ the path obtained by starting at the marked point, following the diagonal up to the primal vertex, then following the UST branch from that vertex to $\partial D$, then the path $\partial D$ up to $x_\r$ and then $\gamma_\r$. The winding field of the UST is then by definition the function $f \to W(\gamma_f, \gamma_f(0))$.

\begin{prop}\label{prop:height_winding}
	There exist a choice of reference flow in the definition of the height function such that, in the bijection of \cref{thm:bijection}, the height function of a dimer configuration is equal to the winding field of the corresponding tree divided by $2\pi$. 
\end{prop}

\begin{remark}
	Since $\gamma_f$ in the above construction is always a path going to infinity, the winding of $\gamma_f$ around its endpoint must be $0$ so $W(\gamma_f, \gamma_f(0))$ is the natural generalisation to possibly rough curves of the intrinsic winding of $\gamma_f$.
\end{remark}

The second map between uniform spanning trees and the dimer model occurs on a special type of graphs called T-graphs and its construction is more involved than the above one so we will not give any detail. For our purpose, it is enough to point out that there is a direct analogue of \cref{thm:bijection,prop:height_winding}. Crucially it is still true in that case that the winding field of the UST is the height function of the associated dimer model. We refer to \cite{BLRannex} for details.

\begin{remark}\label{rq:winding_modulo2pi}
	By definition of the winding, $\gamma \to W(\gamma, x)$ depends continuously on the path $\gamma$ as long as $x \notin \gamma$ and it is not difficult to see that $\gamma \to W( \gamma, \gamma(0))$ also changes continuously when $\gamma$ is moved away from its starting point. On the other hand thanks to our convention that $\gamma_{\r}$ merges with the horizontal axis at some point, we see that for any face $f$, $W( \gamma_f, \gamma_f(0)) \equiv - \arg \gamma_f'(0) [2\pi]$ which actually does not depend on the spanning tree. Indeed $\gamma_f'(0)$ is determined only by the choice of diagonal inside $f$. Combining these two observations, we obtain that $W(\gamma_f, \gamma_f(0))$ is invariant under small modifications of $\gamma_f$ away from $f$.
\end{remark}

\section{Upper bound on $\frac{d \nu_1}{d \nu_2}$}\label{sec:upper_bound}

We now start the proof of \cref{thm:main}. Recall from the introduction that we consider three simply-connected domains $U, D_1, D_2$ with $\overline U \subset D_1 \cap D_1$ and that  $\mu_1 = \mu_1^\d$ and $\nu_1 = \nu_1^\d$ denote respectively the law of UST in $D_1^\d$ and its marginal obtained by only looking at the restriction of the UST to $U$, with similar notations for $D_2^\d$. From now on and until the end of \cref{sec:lower_bound}, we fix two such domains $D_1$ and $D_2$.
We can assume without loss of generality that $D_1 \subset D_2$ since otherwise we can just compare each domain to $D_1 \cup D_2$ and we also first assume that $D_1$ and $D_2$ have smooth boundaries and defer to \cref{continuity_bounded} the extension to domains with rough boundaries.

 With these assumptions in mind, we start in this section with the first bound under the $\nu_1$ measure from \cref{thm:main} where the upper bound is actually the only non-trivial part.


\subsection{Coupling}

 The upper bound $\nu_1 ( \frac{d \nu_1}{d \nu_2} \leq C ) \geq \epsilon$ means that any event (on $U$) which can happen in the smaller domain $D_1$ can also happen with a somewhat similar probability in the bigger domain $D_2$. To prove such a statement, the main step will be to couple $\mu_1$ with a (not too strongly) biased version of $\mu_2$.

\begin{prop}\label{prop:upper_bound_coupling}
	For all $\epsilon>0$, there exists $p>0$ such that for all $\delta$ small enough, there exists an event $E^\d$ such that 
\[
 d_{TV} \Big( \nu^\d_1, \nu^\d_2( . | E^\d)\Big) \leq \epsilon \qquad \text{and} \qquad \mu^\d_2 ( E^\d) \geq p.
\]
\end{prop}
With a slight abuse of notation, in the above statement we denote by $\nu^\d_2 (.  | E^\d)$ the law of the restriction to $U$ of $\mu^\d_2 ( . | E^\d)$ even if $E^\d$ is not measurable with respect to this restriction.

The first step of the proof is to define the event $E^\d$. Recall that $D_1$ is a locally connected set so  $\partial D_1$ can be parametrised by a curve, which we still call $\partial D_1$ with a slight abuse of notation. Fix $\delta > 0$ and take $x_0 \in D_2^\d$ such that $|\partial D_1(0) - x_0| < r$ and let $Y^0$ denote the branch in the UST of $D_2$ starting at $x_0$ (parametrised in $[0, 1]$ so that $Y^0(1) \in \partial D_2^\d$). For any $r > 0$ (to be chosen small later), let $E_r$ be the event defined as follows : there exists $t >0$ such that
\begin{itemize}
	\item As curves up to re-parametrisation, $Y^0([0, t])$ and $\partial D_1$ are at distance at most $r$,
	\item $Y^0[t, 1] \subset D_2^\d \setminus D_1^\d$.
\end{itemize}
By \cref{lem:RWfollow,beurling}, $\P( E_r) > p$ for some $p > 0$ depending on $r$ but not of $\delta$ (as long as it is small enough).

We now introduce a coupling of $\mu_1$ and $\mu_2( . | E_r)$ as follows. First we sample the branch from $x_0$ conditioned on $E_r$. Then we sample the rest of the trees by Wilson's algorithm using the same sequence of starting points  $x_1, x_2, \ldots$ for the walks in both domains. More precisely, for each starting point $x_j$ we consider an independent random walk $X^j$ up to its first exist of $D_2$ (which also determines of course the walk up to its exit from $D_1$). We can use these walks to build either a spanning tree of $D_1^\d$ or $D_2^\d$ by just considering different stopping times so globally we obtain a coupling of $\mu_1$ and $\mu_2( . | E_r)$. In other word, we see Wilson's algorithm in any domain as a function of random walk trajectories and we just apply two of these functions to the same set of trajectories. Let us denote by $\cT_1$ and $\cT_2$ the UST of $D_1^\d$ and $D_2^\d$ respectively under this coupling. We will denote by $\cT_i^{j}$ the partial tree generated by the first $j$ branches in $D_i$, with the convention that $\cT_1^0 = \partial D_1^\d$ and $\cT_2^0$ is the union of 
$\p  D_2^\d$ and the branch from $x_0$. Similarly we let $Y_i^j$ denote the loop erasure of the $j$th walk in the algorithm in $D_i$, i.e $Y_i^j = \cT_i^j \setminus \cT_i^{j-1}$.

\begin{lem}\label{lem:subtree}
	For all $\epsilon > 0$, there exists $C> 0$ such that for all $n \geq 1$, for all $r>0$ small enough and $\delta$ small enough, for all sequence of points $x_j \in D_1^\d \cap D_2^\d$ the above coupling satisfies
	\[
	\P( \cT^n_1 1_{\tilde D_{C^n r}} = \cT^n_2 1_{\tilde D_{C^n r}} ) \geq 1 - n\epsilon
	\]
	where $\tilde D_{r} = \{ x \in D_1 : d(x, \partial D_1) \geq r\}$.
\end{lem}
\begin{proof}
		We proceed by induction.
	First for $n = 1$, let $\tau$ be the first time where $X^1$ hits $\cT_1^0 \cup \cT_2^0$, evidently the loop-erasures agree up to that time. Suppose first that $X^1(\tau) \in \cT_1^0$. By definition of the event $E_r$, $X^1(\tau)$ is within distance $r$ of $\cT_{2}^0$. By Beurling's estimate (\cref{beurling}), we can find $C$ large enough (and independent of everything else) so that with probability $1-\epsilon$, the random walk will hit $\cT_{2}^0$ without exiting $B(X^1(\tau), Cr)$. Note that this also shows that with probability $1-\epsilon$, after the first exit from $\tilde D_{Cr}$ the walk does not come back to $\tilde D_{C^2 r}$. When both events occur, any loop erased after $\tau$ must have a point outside $\tilde D_{Cr}$ but any such loop cannot exit $\tilde D_{C^2 r}$ and of course any piece of the loop-erased walk added after $\tau$ is in $B(X^1(\tau), Cr)$, in particular we have $\cT_1^1 1_{\tilde D_{C^2 r}} = \cT^1_{2} 1_{\tilde D_{C^2 r}}$ as desired (with $2\epsilon$ and $C^2$).

Now consider the case where $X^1 (\tau) \in \cT_2^0$, and note that we must have $X^1(\tau) \in Y^0 \cap D_1^\d$ since by definition $X^1$ has not yet exited from $D_1^\d$. Again by definition of $E_r$, $X^1(\tau)$ must be within distance $r$ from $\partial D_1$ so we are back in the same situation as above.

For $n > 1$, we argue exactly as for the $n=1$ case. The only difference is that if $X^n$ hits $\cT^{n-1}$ before exiting $\tilde D_{C^{n-1} r}$ there is nothing to prove.
\end{proof}

We are now ready to finish the proof of the proposition.

\begin{proof}[Proof of \cref{prop:upper_bound_coupling}]
	Fix $\epsilon_0 > 0$ and $\eta > 0$ such that for all $x \in U$ and $y \in \partial D_1, d( x, y ) > 2 \eta$.
	
	By the finiteness theorem \ref{thm:finitude}, we can find $n$ and $x_1, \ldots x_n$ such that, with probability $1-\epsilon_0/2$, after the first $n$ steps no remaining random walk used in Wilson's algorithm will reach distance $\eta$ from $U$.
	
	We then apply \cref{lem:subtree} with $\epsilon = \epsilon_0/2n$ and $r$ sufficiently small so that $C^n r < \eta$. \cref{lem:subtree} shows that $\P( \cT^n_1 1_{\tilde D_\eta} = \cT^n_2 1_{\tilde D_\eta} | E_r) \geq 1-\epsilon_0/2$. Recalling the finiteness theorem, this shows that $\P( \cT_1 1_{U} = \cT_2 1_{U} | E_r) \geq 1-\epsilon_0$ as desired since as noted just after its definition the probability of $E_r$ can be bounded independently of $\delta$.
\end{proof}

\subsection{Radon-Nikodym bound}

We now derive from the coupling of the previous section a bound of the form used in \cref{thm:main}. 

\begin{lem}\label{coupling_to_radon}
	Let $\nu_1, \nu_2, \tilde \nu_2$ be three measures such that, for some $\epsilon, p > 0$,
	\[
	d_{TV}( \nu_1, \tilde \nu_2 ) \leq \epsilon, \quad \text{and} \quad \frac{d \tilde \nu_2}{d \nu_2} \leq \frac{1}{p}.
	\]
	Then we have
	\[
	\nu_1  ( \epsilon \leq\frac{d \nu_1}{d\nu_2} \leq \frac{2}{p}) \geq 1-5\epsilon.
	\]
\end{lem}
\begin{proof}
	First the lower bound is trivial, indeed 
	\[
	\nu_1 ( \frac{d \nu_1}{d\nu_2} < \epsilon ) = \int  \frac{d \nu_1}{d\nu_2} 1_{ \frac{d \nu_1}{d\nu_2} < \epsilon} d \nu_2 \leq \epsilon.
	\]
	
	For the upper bound, we first note that interpreting the total variation distance as the $L^1$ norm on densities, we get
	\[
	\int  \abs{\frac{d \tilde \nu_2}{d \nu_1} - 1}  d\nu_1 \leq 2 \epsilon.
	\]
	By Markov, $\nu_1( \abs{\frac{d \tilde\nu_2}{d \nu_1} -1} \geq \frac{1}{2} ) \leq  4 \epsilon$ so $\nu_1( \frac{d \tilde\nu_2}{d \nu_1} \leq \frac{1}{2} ) \leq  4 \epsilon$, or in other words
	\[
	\nu_1 ( \frac{d \nu_1}{d \tilde \nu_2} \leq 2) \geq 1 - 4 \epsilon.
\]
We conclude by the assumption $\frac{d \tilde \nu_2}{d \nu_2} \leq \frac{1}{p}$ and union bound.
\end{proof}

Applying \cref{coupling_to_radon} with $\tilde \nu_2 = \nu_2 ( . | E)$ for the event $E$ given by \cref{prop:upper_bound_coupling}, we immediately obtain the first half of \cref{thm:main}:
\begin{prop}\label{prop:upper_bound}
	For any $\epsilon > 0$ there exists $C > 0$ such that for all $\delta$ small enough,
	\[
	\nu_1 ( \frac{1}{C} \leq \frac{d \nu_1}{d\nu_2} \leq C) \geq 1-\epsilon.
	\]
\end{prop}

This last statement can be also rewritten as an inequality involving the probabilities under both measures for any event.

\begin{prop}\label{prop:upper_bound_event}
	There exists two increasing functions $f$ and $g$ from $(0,1)$ to itself such that for any event $A$, we have the following inequality,
	\[
	g(\nu_1(A)) \leq \nu_2(A) \leq f(\nu_1(A)).
	\]
\end{prop}
\begin{proof}
	For $\ep >0$, let $C_\epsilon$ be the constant from \cref{prop:lower_bound} and let $A$ be any event. We define the event $E_{\ep} := \{\frac{1}{C_\epsilon} \leq \frac{d \nu_1}{d\nu_2} \leq C_\epsilon\}$ where $C$ is as in \cref{prop:upper_bound}. 
	By union bound we have $\nu_1(A) \leq \ep + \nu_1(A)$ which leads us to the inequality $\nu_1(A) \leq \ep C_\epsilon\nu_2(A)$ by \cref{prop:upper_bound}. \\
	Defining $G(x):= \inf_{\ep} \{ \ep + C_\epsilon x\}$ and $g(x):= \inf\{u, G(u) \geq x\}$ we have the first half of the inequality,
	\[
	g(\nu_1(A)) \leq \nu_2(A)
	\]
	In a similar fashion, we have by union bound and \cref{prop:upper_bound} the upperbound $\nu_2(A) \leq \nu_2(E_{\ep}^{c}) + C_\epsilon \nu_1(A)$. Leading us to the second part of the inequality 
	\[
	\nu_2(A) \leq f(\nu_1(A))
	\]
	in which $f(x):= \inf_{\ep} \{\nu_2(E_{\ep}^{c}) + C_\epsilon x\}$
\end{proof}

\section{Lower bound on $\frac{d \nu_1}{d \nu_2}$}\label{sec:lower_bound}

We now turn to the second inequality from \cref{thm:main}. As in the previous section the main point is the following coupling statement.

\begin{prop}\label{prop:lower_bound_coupling}
	For all $\epsilon > 0$, there exists $p > 0$ depending only on $U, D_1, D_2$ such that for all $\delta$ small enough, there exists a law $\tilde \nu_1$ such that
	\[
	d_{TV}(\nu_2, \tilde \nu_1 ) \leq \epsilon, \qquad \frac{d \tilde \nu_1 }{ d \nu_1} \leq \frac{1}{p}.
	\]
\end{prop}
We think of this proposition as an analogue of \cref{prop:upper_bound_coupling} because we see  $\tilde \nu_1$ as a kind of conditioned version of $\nu_1$, with a conditioning event of probability at least $p$. This case is however significantly more involved than the previous one so $\tilde \nu_1$ is not exactly obtained by a conditioning.

Let us highlight the main steps of the proof. As before, we first focus on a single branch. The challenge compared to the previous section is that is no longer possible to simply turn the larger domain into a close approximation of the smaller one by an appropriate event. Instead, given a random walk in the large domain $D_2$, we want to create a path in $D_1$ which agrees with the first one on $U$ (so that their loop-erasure match) and reproduces the same topology with its excursions outside of $U$ (so that any erasure coming from the excursions outside of $U$ matches).

\begin{figure}[ht]
	\begin{center}
		\includegraphics[width=.5\textwidth]{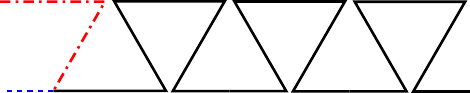}
		\caption{An example where two coupled trajectories lead to different loop-erasures. Consider the two paths obtained by following one of dashed colored lines followed by the continuous black line. The points where the paths come very close to itself should be though as actual intersection and are only drawn with a small gap for readability. The loop erasures of the path starting red is the upper straight line while for the blue path it is the bottom straight line. Note that in both case, all erased loops are bounded.}
		\label{fig:uncloupledLE}
	\end{center}
\end{figure}

The first step is to clarify how to go from a coupling of two random walks to their loop-erasures (see \cref{fig:uncloupledLE} for why this is non-trivial). Here the main idea is to introduce a loop-erasing procedure with both forward and backward aspects, designed to ensure that the loop-erasure are obtained ``locally'' from the random walks and must therefore agree whenever the walks agree.

The second step is actually a trick to get rid of all the topological complexity from possible large scale loop erasure : By \cref{walk_given_erasure}, the conditional law of a random walk $X$ given that its loop erasure is some simple path $\gamma$ is obtained by adding to $\gamma$ loops sampled from an independent loop soup.
By \cref{loop_mass}, with positive probability all these added loops are small. On that even all random walks are guarantied to remain within a small distance of their erasures and no large scale complexity can happen from the erasing procedure.

Finally, as in \cref{sec:upper_bound}, we will focus on the first loop-erased walk in Wilson's algorithm. The proof for any finite number will be similar and an application of Schramm's finiteness theorem will conclude the proof by going from a finite number of branches to the full tree.

\subsection{Mixed Loop-erasing}\label{sec:mixed_erasing}

Denote by $\fLE(X)$ and $\bLE(X)$ the forward and backward loop-erasure of a path $X$ and by $\oplus$ the concatenation of paths.

\begin{defin}\label{def:mixedLE}
	Let $X = (X_t)_{0 \leq t \leq T_{i_{\max}}}$ and $(T_i)_{1 \leq i \leq i_{max}}$ be respectively a path and a sequence of times with $T_i < T_{i+1}$ almost surely. We define the mixed loop erasure of $X$ with respect to the sequence $T_i$ by the following inductive procedure (see \cref{fig:mixedLE} for an illustration):
	\begin{itemize}
		\item We define $Y_1 := \bLE(X[ 0, T_1] )$.
		\item Given $Y_i$ a simple path, let $s_i = \inf \{ s : Y_i(s) \in X[T_i, T_{i+1}]\}$ and let $t_i = \max \{ t \in [T_i, T_{i+1}] : X(t) = Y(s_i)\}$. In other words, $Y_i(s_i)$ is the first point where $Y_i$ intersect the trajectory of $X$ and $t_i$ is the (last) time in $X$ where this intersection occurs.
		\item We set $Y_{i+1}:= Y_i[0, s_i] \oplus \bLE(X[t_i, T_{i+1}] )$.
		\item we define the mixed loop-erasure to be the final simple path $Y_{i_{\max}}$.
		\end{itemize}
\end{defin}

\begin{figure}[ht]
	\begin{center}
		\hspace{.05\textwidth}
		\includegraphics[width=.3\textwidth]{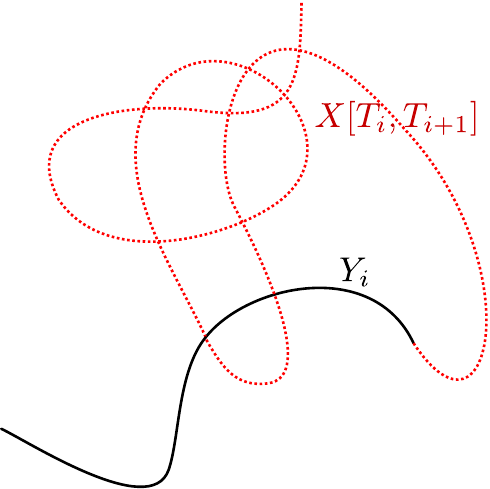}
		\includegraphics[width=.3\textwidth]{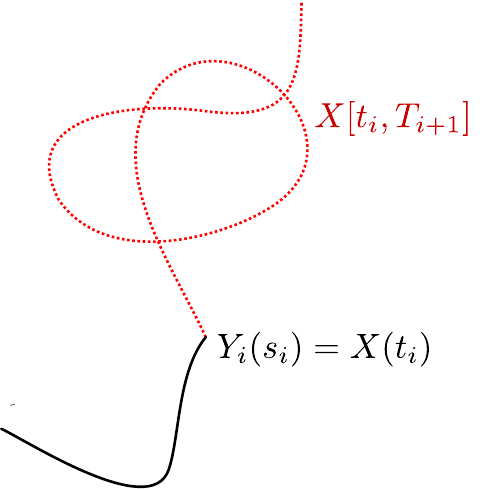}
		\includegraphics[width=.3\textwidth]{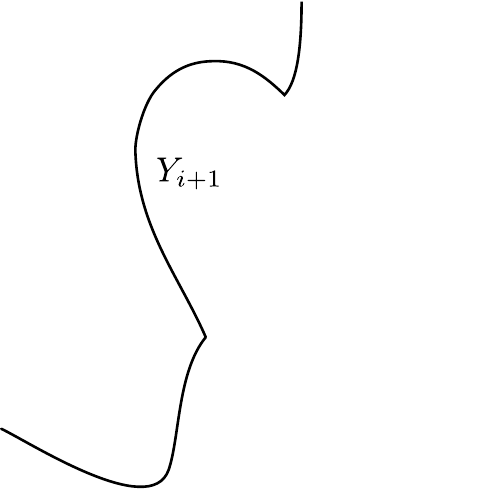}
		\caption{One step in the mixed loop-erasing procedure. Left : the ``current'' loop erasure $Y_i$ in dark complete line and the additional portion of walk $X[t_i, T_{i+1}]$ in dashed red. Center : $Y[0, s_i]$ and $X[t_i, T_{i+1}]$; in a sense the forward loop from $Y_i(s_i)$ to $X(t_i)$ has been erased. Right : $Y_{i+1}$ obtained by doing a reverse loop-erasure. Note that in this example, the output of the mixed loop erasure is different from either the forward or backward one.}
		\label{fig:mixedLE}
	\end{center}
\end{figure}

It is well known that the forward and backward loop-erasure of a random walk have the same law. For the mixed loop-erasing, we cannot give such a general statement but we nonetheless have  the following.
\begin{lem}
	Let $G$ denote a graph with a boundary $\partial G$ and let $X$ denote a random walk stopped when it first hits $\partial G$. Let $E_i$ be a sequence of subsets of $G$, let $T_i$ be the sequence of stopping times defined by $T_{i+1} = \inf\{ t > T_i : X_t \in E_{i+1} \cup \partial G \}$ and let $i_{\max}$ be the first $i$ such that $T_i \in \partial G$, i.e $T_{i_{\max}} = \inf \{ t : X_t \in \partial G\}$ is the time where we stop the walk. The mixed loop erasure of $X$ with respect to the times $T_i$ of $X$ has the law of the (forward) loop erasure of $X$.
\end{lem}
\begin{proof}
	We claim that the sequence of paths $( Y_i)_{0 \leq i \leq i_{\max}}$ is equal in law the the sequence $(\fLE( X[0, T_i]))_{0 \leq i \leq i_{\max}}$. Indeed suppose by induction that this is true for some $i$, fix some path $\gamma$, and let us compare the law of $Y_{i+1}$ given $Y_i=\gamma$ to the law of $\fLE( X[0, T_{i+1}])$ given $\fLE( X[0, T_{i}])=\gamma$. 
		Also define $s_i = \inf \{ s : \gamma(s) \in X[T_i, T_{i+1}]\}$ and $t_i = \max \{ t \in [T_i, T_{i+1}] : X(t) = \gamma(s_i)\}$ following the notation from \cref{def:mixedLE}. 
	
	Note that $\fLE( X[0, T_{i+1}]) = \gamma[0,s_i] \oplus \fLE( X[t_i, T_{i+1}])$ and that, conditionally on $s_i$, $X[t_i, T_{i+1}]$ is a random walk conditioned to hit $E_{i+1}$ before $\gamma$. By Lemma 7.2.1 in \cite{Lawler2012}, there is a measure-preserving bijection $\varphi_i$ on such excursions paths satisfying 
	\[
	\bLE( \varphi_i( X[t_i, T_{i+1}])) = \fLE( X[t_i, T_{i+1}]).
	\] Note that this includes the fact that $t_i$ and $T_{i+1}$ are invariant under $\varphi_i$. In particular
	\begin{multline*}
	Y_{i+1} := \gamma[0, s_i] \oplus \bLE( X[t_i, T_{i+1}])  \overset{\text{law}}{=}  \gamma[0, s_i] \oplus \fLE( X[t_i, T_{i+1}])  = \fLE( X[0, T_{i+1}]).
	\end{multline*}
\end{proof}

From the proof, we also have the following corollary
\begin{corollary}\label{cor:loop-reversal}
	Given sets $E_i$ (with associated hitting times $T_i$) as above, there exists a bijective, measure preserving, map $\varphi$ from the set of random walks trajectory to itself such that the forward loop-erasure of $X$ and the mixed loop-erasure of $\varphi(X)$ with respect to the times $T_i$ are identical.
\end{corollary}
With a slight abuse of terminology, we will speak in the following of the mixed loop-erasure with respect to a sequence of sets in the above context.

\subsection{RW in $D_2$}\label{sec:RWD2}

In this section, as mentioned in the beginning of the section, we turn to the construction of a RW path $\tilde X_2$ whose mixed loop-erasure has exactly the law of a branch of UST in $D_2$ and yet never creates any large loop. As an aside on notations, we will use unmarked letters to describe objects with the original law of random walk/UST/\ldots and a tilde to denote biased versions used to construct the measure $\tilde \nu_1$. Without loss of generality, we will start our walk at $0$ to simplify notations.

From now on we fix three open connected sets $U_1, U_2, U_3$ with $\overline{U} \subset U_1$, $\overline{U_i} \subset U_{i+1}$ and $\overline U_3 \subset D_1 $ and we let $E_i$ denote the alternating sequence of sets $E_{2i} = U_2^\d$ and $E_{2i+1} = D_1 \setminus U_3^\d$. Fix $r$ to be chosen small enough later (in particular so that the boundaries of $U, U_1, U_2, U_3, D_1, D_2$ are all at distance more than $r$) and we consider the event $G = G_r$ that all loops in the loop soup  $(\ell_k)_{k \geq 1}$ have diameter at most $r$.
Let $Y_2$ be a loop-erased random walk from $0$ in $D_2^\d$, we define $X_2$ and $\tilde X_2$ by sampling respectively an unbiased loop soup or one conditioned on $G_r$ independent from $Y_2$, then applying \cref{cons:walk_given_erasure} and finally applying the bijection of \cref{cor:loop-reversal} relative to the sequence of sets $E_1$.

\begin{lem}\label{lem:radonX2}
	$Y_2$ is the mixed loop erasures of both $X_2$ and $\tilde X_2$ with respect to the sets $E_i$. Furthermore the pair $(X_2, Y_2)$ has the law of a random walk and its loop-erasure. Finally there exists $C$ depending only on $D_2$ such that for all $\delta$ small enough and for all paths $\chi$ on $D_2^\d$
	\[
	\P( \tilde X^2 = \chi | Y_2 ) \leq C \P( X^2 = \chi | Y_2).
	\]
\end{lem}
\begin{proof}
	The first two points are trivial by construction and because the map from \cref{cor:loop-reversal} is a measure preserving bijection. The last point is because by \cref{loop_mass}, the even $G_r$ has a strictly positive probability independently of $\delta$.
\end{proof}

Finally, as mentioned above the reason why we introduce the conditioning by $G_r$ is to get rid of possible large scale erasures. For this we want to say that in a sense the $r$ neighbourhood of $\chi$ is ``simple'' with high probability.
\begin{defin}[Definition 3.3 in \cite{Schramm2000}]
	An $r, R$-quasiloop of a path $\chi$ is a pair $a, b$ such that $\abs{\chi(a) - \chi(b)} \leq r$ but $\diam( \chi[a, b]) \ge R$.
\end{defin}

\begin{lem}[Lemma 3.4 in \cite{Schramm2000}]\label{quasi_loop}
	For all sequence of starting points, for all $R > 0$
	\[
	\lim_{r \to 0} \limsup_{\delta \to 0} \P( Y \text{ has an $r,R$-quasiloop}) = 0.
	\]
\end{lem}

\subsection{RW in $D_1$}\label{sec:RWD1}

In this section we define the walk in $D_1$ and control that its law is not too far from that of a simple random walk. Note that the construction will depend on $U_1, U_2, U_3$ which were previously defined, the distance $r$ that was used to define $X_2$ and an integer $N$. We also fix a diffeomorphism $\phi$ from $D_2$ to $D_1$ such that the restriction of $\phi$ to $U_3$ is the identity.

Define a sequence of times $T_i = T_i^2$ inductively by $T_0 = 0$ and $T_{2i+1} = \inf \{ t \geq T_{2i} , X^2(t) \not \in U_3 \}$, $T_{2i} = \inf\{ t \geq T_{2i-1} : X^2(t) \in U_2 \}$.  We also let $T_{\max}$ denote the hitting time of $\partial D_2$ by $X^2$ and we let $i_{\max} = \max \{ i : T_{i} < T_{\max} \}$  i.e $i_{\max}$ is the last time in the sequence where we are on $\partial U_3$. We will denote by $T^1$ and $T^Y$ the corresponding set of times for either $X^1$ (when it will be defined) and $Y_2$.

We say that the path $Y_2$ is \emph{$\epsilon$-good} is it satisfies the following:
\begin{itemize}
	\item It does not have any $(\epsilon^2,\epsilon)$-quasiloop;
	\item $i_{\max}^Y \leq 1/\epsilon$;
	\item For every $i$, let $T_{2i^-}^Y$ be the first time after $T_{2i-1}^Y$ where $Y$ comes within distance $\epsilon^2$ of $U_2$, we want that $Y[T_{2i^-}, T_{2i}] \subset B( Y(T_{2i^-}), \epsilon)$.
\end{itemize}
Combining the probability to see quasi-loops from \cref{quasi_loop}, the Beurling estimate from \cref{beurling} and the fact that by uniform crossing $i_{\max}$ has an exponential tail, we see that the probability that $Y_2$ is $\epsilon$-good goes to $1$ as $\epsilon \to 0$, uniformly over the mesh size $\delta$.

We define the path $\tilde X^1$ conditionally on $\tilde X^2$ (and therefore $Y^2$) by concatenating different pieces of random walk and recall that we denote by $\tilde X[t, t']$ the subpath between times $t$ and $t'$.
\begin{itemize}
	\item If $Y_2$ is not $\epsilon$-good, then we just let $\tilde X^1$ be a random walk independent of everything else. From now on assume that $Y_2$ is $\epsilon$-good and for ease of notation assume that $i_{\max} \geq 3$.
	\item  We set $\tilde X^1[0, T_1^1] := \tilde X^2[0, T_1^2]$.
	\item  We choose a sequence of rectangle following $\phi(\tilde X_2[T_1^2, T_2^2])$ up to a precision $r$ (see \cref{sec:assumptions} ) and such that only the last one comes within distance $\epsilon^2/2$ of $U_2$. We let $\tilde X^1[T_1^1, T_2^1]$ be a random walk started from $\tilde X^1(T_1^1) = \tilde X^2(T_1^2)$, conditioned to first hit $U_2$ at the position $\tilde X^2(T_2^2)$ and to cross all the rectangles in the above sequence in order.
	\item We set $\tilde X^1( T_2^1, T_3^1) := \tilde X^2( T_2^2, T_3^2)$ which is possible since by construction $\tilde X_2(T^2_2) = \tilde X^1(T_2^1)$.
	\item We iterate the construction up to $i_{\max}$. From $T_{i\max}$ we complete $\tilde X^1$ by a random walk following $\phi(\tilde X_2[T_{i_{\max}}, T_{\max}])$ and exiting $D_1$. 
\end{itemize}

\begin{lem}\label{lem:RN_X1}
	For all $\epsilon > 0$ and $r >0$ such that $r \leq \epsilon^2$, there exists some $C > 0$ such that for all $\chi$
	\[
	\P(\tilde X^1 = \chi ) \leq C \P( X^1 = \chi )
	\]
	where $ X^1$ is a random walk stopped when it exits $D_1$.
\end{lem}
\begin{proof}
	First by construction, if $Y_2$ is not $\epsilon$-good there is nothing to prove so let us analyse $\P( \tilde X_1 = \chi | Y_2)$ in the case where $Y_2$ is $\epsilon$-good. To construct $\tilde X_1$, we need to know $Y_2$ and all the pieces of trajectory $\tilde X^2[T_{2k}, T_{2k+1}]$ but conditionally on that information we sample independent random walks, therefore
\begin{multline*}
	\P (\tilde X_1 = \chi | Y_2 )  =  \P\Big( \forall k < \tfrac{i_{\max}}{2},\,  \tilde X_2[T_{2k}, T_{2k+1}] = \chi[T_{2k}, T_{2k+1}] \Big| Y_2\Big)  \\
	 \times	 \prod_{k \leq \frac{i_{\max}}{2}} \P_{\chi(T_{2k-1})} \Big(  X[0, T] = \chi[T_{2k-1}, T_{2k}]  
	  \Big|  X (T)  = \chi( T_{2k}) , \text{crossing rectangles}, Y_2 \Big) \\
	 \times \P_{\chi( T_{i_{max}})} \Big( X[0, T] = \chi[T_{i_{\max}}, T_{\max}] \Big|  \text{crossing rectangles}, Y_2 \Big),
\end{multline*}
where in the second line we denote by $ X$ a simple random walk started at the position $\chi(T_{2k-1})$ indicated in the probability and by $T$ the hitting time of $U_2$ while the conditioning event is given in the third step in the definition of $\tilde X_1$. In the third line we use similar notations referring to the last step.

We start by analysing the terms in the second line, i.e from $2k-1$ to $2k$. By the lack of $(\epsilon^2,\epsilon)$-quasi loop in $Y$, we can find an uniform bound on the number of rectangles that have to be crossed to follow $\phi(Y)$. By the conditional uniform crossing, each rectangle before the last has uniformly positive probability to be crossed by the random walk (since they are all at distance at least $\epsilon^2/2$ of $\partial U_2$). 
 Therefore, up to a multiplicative factor, we can forget about the conditioning to cross rectangles in the second line but note that this also removes any dependence on $Y$. 
 For the last step (i.e the third line in the equation), by uniform continuity of $\phi$ we have a uniform bound on the number of rectangles that have to be crossed to follow $\phi(Y)$ from $\chi( T_{i_{\max}})$ up to a very small distance of $\partial D_1$ and from there \cref{beurling} shows that the walk is likely to exit $D_1$ without going far away.
  Overall if we also integrate to remove the conditioning on $Y_2$, we obtain
\begin{multline*}
\P (\tilde X_1 = \chi )  \leq C \E_{Y_2} \Big[ \P( \forall k < i_{\max}/2,  \tilde X_2[T_{2k}, T_{2k+1}] = \chi[T_{2k}, T_{2k+1}] \mid Y_2) \Big] \\
\times	 \prod_{k \leq \frac{i_{\max}}{2}} \P_{\chi(T_{2k-1})} \Big( X[0, T] = \chi[T_{2k-1}, T_{2k}]  
\Big|  X (T)  = \chi( T_{2k}) \Big) \\
\times \P_{\chi( T_{i_{max}})} \Big(  X[0, T] = \chi[T_{i_{\max}}, T_{\max}] \Big) \Big),
\end{multline*}
By \cref{lem:radonX2} when conditioned on $Y_2$, the law of $\tilde X_2$ is absolutely continuous with respect to a random walk conditioned by its loop-erasure. When taking the expectation over $Y_2$, we therefore recover up to constant the law of a random walk, i.e we have 
\begin{multline*}
\P (X_1 = \chi )  \leq C  \P( \forall k < i_{\max}/2,  X[T_{2k}, T_{2k+1}] = \chi[T_{2k}, T_{2k+1}] )  \\
\times	 \prod_{k \leq \frac{i_{\max}}{2}} \P_{\chi(T_{2k-1})} \Big( X[0, T] = \chi[T_{2k-1}, T_{2k}]  
\Big|  X (T)  = \chi( T_{2k}) \Big) \\
\times \P_{\chi( T_{i_{max}})} \Big( X[0, T] = \chi[T_{i_{\max}}, T_{\max}]\Big).
\end{multline*}
Now for the random walk, it is not hard to see that 
\begin{multline*}
\P\Big( \forall k < i_{\max}/2,   X[T_{2k}, T_{2k+1}]  = \chi[T_{2k}, T_{2k+1}] \Big) \leq C \textrm{harm}_{D_2 \setminus U_2} ( \chi(T_{i_{\max}}) , \partial D_2)\\ 
\times  \prod \P_{\chi(T_{2k})} \Big(  X[0, T] =\chi[T_{2k}, T_{2k+1}] \Big)   \textrm{harm}_{D_2 \setminus U_2}( \chi(T_{2k-1}), \chi(T_{2k})) ,
\end{multline*}
where $\textrm{harm}_{D_2 \setminus U_2}( \chi(T_{2k-1}), \chi(T_{2k}))$ is the probability starting from  $\chi(T_{2k-1})$ to reach $U_2$ before exiting $D_2$ and to do so through $ \chi(T_{2k})$. 
When combining the previous two equations, the harmonic measure in $D_1$ and $D_2$ are within a constant of each other and they compensate the conditioning on the endpoints for odd intervals. We conclude by comparing the expression for simple random walk and for $X_1$.
\end{proof}

\subsection{loop-erasure}\label{sec:erasure}

Now that we have our candidate random walk in $D_1$, we can look at its loop-erasure. Let $\tilde Y_1$ denote the mixed loop erasure of $\tilde X_1$ with respects to the alternating sequence of sets $E_{2i} = U_2^\d$ and $E_{2i+1} =D_1^\d \setminus U_3^\d$.
\begin{lem}\label{lem:erasure}
	There exists $C > 0$ such that for all $\epsilon >0$ if $r$ is small enough compared to $\epsilon$ and if $Y_2$ is $\epsilon$-good then we have
	\[
	d( \tilde Y_1, \phi(Y_2) ) \leq C \epsilon, \text{ and } \tilde Y_1 \cap U_1 = Y_2 \cap U_1,
	\]
	where recall that $d( \tilde Y_1, \phi(Y_2) )$ denotes the uniform distance up to time parametrisation.
\end{lem}
\begin{proof}
	We start by proving the first statement. By construction, $d(\tilde X_2, Y_2) \leq r$ since $\tilde X_2$ is obtained from $Y_2$ by adding loops of size at most $r$. Also, since $\tilde X_1$ is defined by gluing portions of paths where $\tilde X_2 = \phi(\tilde X_2)$ with portions following $\phi(X_2)$, by construction it satisfies $d( \tilde X_1, \phi( \tilde X_2)) \leq r$.	
  Since $\phi$ is Lipschitz (recall that it was chosen as a diffeomorphism), we therefore have $d( \tilde X_1, \phi( Y_2)) \leq Cr$ for some $C > 0$. Also using the Lipschitz property of $\phi^{-1}$, we see that $\phi( Y_2)$ does not have $c \epsilon^2, C \epsilon$ quasi-loops for some $c, C$ depending only on $\phi$.
	
	Now let us bound the size of the loops created by $\tilde X_1$. Suppose by contradiction that $\tilde X_1$ creates a loop of diameter $K \epsilon$ for some $K$ to be chosen large enough later, and let $u, v, w$ be three times such that $v \in [u, w]$, $\tilde X_1(u) = \tilde X_1(w)$ and $d(\tilde X_1(v), \tilde X_1(u)) \geq K \epsilon$. 
	Since $d( \tilde X_1, \phi(\tilde X_2)  )\leq r$, there must exists corresponding times $u', v', w'$ such that $d( \phi(Y_2)(u'), \phi( Y_2(w')) \leq 2Cr$ and $d( \phi(Y_2)(u'), \phi( Y_2(v')) \leq K \epsilon - 2Cr$. In particular $\phi( Y_2)$ must have an $(2Cr, K\epsilon-2Cr)$-quasiloop which is a contradiction for $r$ small enough and $K$ large enough. As a corollary, $d(\tilde Y_1, \tilde X_1) \leq K \epsilon$ which completes the proof of the first point. From now on we assume that $\epsilon$ is chosen so that $K \epsilon$ is small enough compared to the distances between the boundary of any of the $U_i$ or $D_i$.
	
	For the proof of the second point, let $\tilde Y_1^i$ denote the mixed loop erasure of $\tilde X_1[0, T_i]$ and similarly for $\tilde X_2$ (this is the same notation as in the definition of the mixed loop-erasure \ref{def:mixedLE} except that we moved the ``number of step'' index up because we now also have the domain index) and let us proceed by induction. For $i=1$, there is nothing to prove, we have $\tilde Y_1^1 =  Y_2^1$. For $i= 2$ (and for general even steps later), by the bound on the size of erased loops, we see that $\tilde X_1[T_1, T_2]$ cannot intersect the portion of $\tilde Y_1$ before its last visit to $U_2$. Since the same holds for $\tilde X_2$ we must in particular have 
	\[
	\tilde Y_1^2 \cap U_1 = \tilde Y_1^1 \cap U_1 =  Y_2^1 \cap U_1 = Y_2^2 \cap U_1.
	\]
	For $i=3$ and general odd steps afterwards, let $s_2^1, s_2^2, t_2^1, t_2^2$ be defined as in \cref{def:mixedLE} so that we have
	\[
	\tilde Y_1^3 = \tilde Y_1^2[0, s_2^1] \oplus \bLE( \tilde X[t_2^1, T_3]) , \quad  Y_2^3 = Y_2^2[0, s_2^2] \oplus \bLE( \tilde X[t_2^2, T_3]),
	\]
	where note that crucially we do not need a subscript for the random walk trajectory $\tilde X$ of the time $T_3$ as they refer to a portion where $\tilde X_1 = \tilde X_2$. By the bound on the size of erased loops, both $\tilde X[T_2, t_2^1]$ and $\tilde X[T_2, t_2^2]$ must stay within $2K\epsilon$ of $\tilde X(T_2)$ and therefore (still using the bound on the size of erased loops)$ \bLE( \tilde X[t_2^1, T_3])$ and $\bLE( \tilde X[t_2^2, T_3])$ must agree whenever they are at distance $3 K \epsilon$ of $\tilde X(T_2)$, in particular $\tilde Y_1^3 \cap U_1 = Y_2^3 \cap U_1$.
	We conclude the proof by induction.
\end{proof}

\subsection{Coupling several paths}\label{sec:multiplepath}

In this section, we conclude the proof of \cref{prop:lower_bound}. From \cref{sec:mixed_erasing,sec:RWD2,sec:RWD2,sec:erasure}, we showed how to construct a single path in $D_1$ that agree over $U_1$ with a loop erased random walk in $D_2$. 
\begin{lem}
	For all open sets $U$, $U_1$ such that $\bar U \subset U_1$ and $\bar U_1 \subset D_1$, for all homomorphism $\phi$ from $D_2$ to $D_1$ restricting to the identity on a neighbourhood of $U_1$, for all $n \geq 1$, there exits $C$ such that for all $x_1, x_n \in U^\d$ there exists a coupling of $Y_2(v_1), \ldots, Y_2(v_n)$ with a law $\tilde \mu_1$ on paths $\tilde Y_1(v_n), \ldots, \tilde Y_1(v_n)$ in $D^\d_1$ satisfying
	\begin{gather*}
	\P( \forall i, \, Y_2( v_i) \cap U_1 = \tilde Y_1( v_i) \cap U_1 ) \geq 1-\epsilon, \\
	\P( \forall i, \, d( \tilde Y_1(v_i), \phi( Y_2(v_i)) \geq \epsilon) \leq \epsilon,\\
	 \frac{d\tilde \mu_1 }{d \mu_1} \leq C.
	\end{gather*}
	where with a slight abuse of notation $\mu_1$ denotes the joint law of the spanning tree branches starting from $v_1, \ldots, v_n$ in $D_1$.
\end{lem}
\begin{proof}
	We prove the result by induction on $n$. For $n=1$, we combine \cref{lem:erasure} to control the first two lines on the event that $Y_2^is $$\epsilon$-good (whose probability goes to $1$ as $\epsilon \to 0$ with \cref{lem:RN_X1} for the third line (the uniform bound on the Radon-Nicodym derivative there trivially extends to the law of any function of $\tilde X_1$). Suppose now that it holds for some $n \geq 1$. We choose $U^{(n)}_1$ such that $\overline{U}_1 \subset U^{(n)}_1$ and such that $\phi$ restricts to the identity in a neighbourhood of $U_1'$. By induction we can find a law on paths $\tilde Y_1(v_1), \ldots, \tilde Y_1( v_n)$ satisfying the assumption for $U_1^{(n)}$ and some $\epsilon_n$ to be chosen small enough later. By Wilson's algorithm, what remains to be done is to couple loop-erased random walks from $v_{n+1}$ in respectively $D_1 \setminus \cup \tilde Y_1(v_i)$ and $D_2 \setminus \cup Y_2( v_n)$. We can also assume without loss of generality that the $\tilde Y_1(v_i)$ paths and the $Y_2(v_i)$ paths agree in $U_1^{(n)}$.
	
	This is very similar to the case of the first walk from \cref{sec:RWD1,sec:RWD2,sec:erasure}, with only two differences : the random walk can stop inside $U_1$ and the boundary of the two (slitted) domains are not exactly mapped to each other by $\phi$ (note that in the statement, $\phi$ must be deterministic it cannot be adapted for the slit domains we have at each step). 
	
	 The first difference causes no issue whatsoever. It is now possible for the sequence of time $T_i$ from the definition of the mixed loop-erasure \ref{def:mixedLE} to end at an odd step if a random walk hits one of the previous paths $Y(v_i)$ before the boundary. This does not break the coupling since over these time intervals $\tilde X^1 = \tilde X^2$ and by induction the paths $Y(v_i)$ agree in $U_1'$. 

	The second difference could be more problematic. For example if the new spanning tree branch $Y^2(v_{n+1})$ comes close to $Y^2(v_1)$ without hitting it, it could happen that $\phi( Y^2(v_{n+1}))$ intersects $\tilde Y^1(v_1)$. In this case a walk following $\phi( Y^2(v_{n+1}))$ might very well hit the boundary ``too soon'' and break the first part of the proposition.
	
	 However this cannot happen if we take $\epsilon_n$ sufficiently small compared to $\epsilon$. Indeed, the first time either $\tilde X_1$ or $\tilde X_2$ come within $\epsilon_n$ of a boundary, the other walk must be within $ C\epsilon_n$ of a boundary and therefore both are very unlikely to move more than $\epsilon$ from their current position and therefore cannot erase any large loop. 
\end{proof}

Similarly to the first case, once a large but finite number of paths have been coupled in a neighbourhood of $U$, we can extend the coupling to the rest of the tree by the finiteness theorem so overall we obtain \cref{prop:lower_bound_coupling}. Using \cref{coupling_to_radon} we obtain the second part of \cref{thm:main}.
\begin{prop}\label{prop:lower_bound}
	For any $\epsilon > 0$, there exists $C > 0$ such that for all $\delta$ small enough
	\[
	\nu_2( \frac{1}{C} \leq \frac{d \nu_1}{d \nu_2} \leq C) \geq 1-\epsilon.
	\]
\end{prop}

\begin{remark}\label{rk:generalD}
	Now that we have upper and lower bound as in \cref{prop:upper_bound,prop:lower_bound} in the case where $D_1 \subset D_2$, the result clearly extends to arbitrary $D_1$ and $D_2$ by just introducing a intermediate common subset to both. In fact in the above proof of \cref{prop:lower_bound} we did not really use the fact that $D_1 \subset D_2$ so we could have omitted the special case treated in \cref{sec:upper_bound}. However the fact that in the setting of \cref{prop:upper_bound} only a simple conditioning is needed is worth noting and starting with the simpler case is useful to the exposition. Also the proof of the simpler case will be useful later for \cref{continuity_bounded}.
\end{remark}

\section{Extensions and generalisations}\label{sec:extensions}

In this section, we provide small variants of our coupling in order to paint a more global picture around \cref{thm:main}.

\subsection{Continuity with respect to the domains}

As mentioned above, we did not keep track of the dependence on the domains of the constants $C$ and $\epsilon$ from \cref{thm:main} but it is a very natural question, the answer of which turns out to be a corollary of the proof of the upper bound (\cref{prop:upper_bound_coupling}). Note that this also covers the case of a non smooth boundary in \cref{thm:main} and shows that the discretization setup used in \cref{sec:assumptions} could be replaced by any other reasonable one.
\begin{corollary}\label{continuity_bounded}
	For any simply connected open sets $U, V$ with $\overline U \subset V$ and for any $\epsilon >0$, there exists $r > 0$ and $\delta_0$ such that for all simply connected domains $D_1, D_2$, if $V \subset D_1 \cap D_2$ and the Hausdorff distance between $\partial {D_1}$ and $\partial D_2$ is less than $r$, then
 \[
\nu_1 ( 1-\epsilon \leq \frac{d \nu_1}{d\nu_2} \leq 1+\epsilon) \geq 1-\epsilon, \qquad \nu_2(  1-\epsilon \leq \frac{d \nu_1}{d\nu_2} \leq 1+\epsilon ) \geq 1-\epsilon.
\]	
\end{corollary}
\begin{proof}
	In the proof of \cref{prop:upper_bound_coupling}, note that the role of the event $E^\d$ is precisely to reduce the problem to the case where $D_1$ and $D_2$ are very close. It is therefore easy to see that \cref{lem:subtree} applies in this setting and that for $r$ small enough one has $d_{TV} (\nu_1^\d, \nu_2^\d) \leq \epsilon^2$. We conclude similarly as in \cref{coupling_to_radon} : by definition of the total variation distance we known that $\int \abs{\frac{d \nu_2}{d \nu_1} -1} d\nu_1 \leq 2\epsilon^2 $, and then we simply apply Markov's inequality.
\end{proof}

In more abstract term, the corollary above says that the constant $C$ in \cref{thm:main} (and even the whole marginal law $\nu$) depends continuously on the domain for the topology defined by the Hausdorff distance on their boundaries and that the continuity is even uniform if the boundary is far away from $U$. The case of the continuity ``close to the full plane'' was actually proved earlier.

\begin{prop}\label{continuity_unbounded}
	For any bounded open set $U$ and $\epsilon > 0$, there exists $R > 0$ and $\delta_0 > 0$ such that for all $\delta < \delta_0$, for all $D_1, D_2$ which are either domains or the full plane, if $B(0, R) \subset D_1$ and $B(0, R) \subset D_2$ then 
	 \[
	\nu_1 ( 1-\epsilon \leq \frac{d \nu_1}{d\nu_2} \leq 1+\epsilon) \geq 1-\epsilon, \qquad \nu_2(  1-\epsilon \leq \frac{d \nu_1}{d\nu_2} \leq 1+\epsilon ) \geq 1-\epsilon.
	\]
\end{prop}
\begin{proof}
	The total variation distance between the two measures is controlled by Theorem 4.21 in \cite{BLR16} if the aspect ratio between $U$ and the domains is large enough. We conclude as in \cref{continuity_bounded}.
\end{proof}

\subsection{Dimer configurations and height functions}

As presented in \cref{sec:winding}, the uniform spanning tree is deeply connected to the dimer model and its height function so it is natural to ask whether our results can be generalised to that setting. In this section we will give a positive answer to that question, with only a slight complication with respect to the regularity assumption on the boundary.

As before consider $D_1, D_2, U$ open, simply connected such that $\bar U \subset D_1 \cap D_2$. In order to consider the dimer and height function on them, further assume that $D_1$ and $D_2$ have a locally connected boundary with a marked prime end (see \cref{sec:winding} for details). We will denote by $\nu_1^{(d)}$ (resp. $\nu_2^{(d)}$) the law of the restriction of the dimer configurations in $D^\d_1$ (resp. $D^\d_2$) to $U$. Similarly we denote by $\nu^{(h)}_1$ and $\nu_2^{(h)}$ the laws of the restrictions of the height function to $h$. Note that $\nu_1^{(d)}$ and $\nu_1^{(h)}$ are different because knowing a dimer configuration on $U$ only determines the height function over $U$ up to a global constant.	
We have a direct extension of \cref{thm:main}
\begin{thm}\label{thm:main_dimer}
	For any $\epsilon$, there exists $C > 0$ such that for all $\delta$ small enough,
	\[
		\nu_{1}^{(d)}( \frac{1}{C} \leq \frac{d \nu_1^{(d)}}{d \nu_2^{(d)}} \leq C)  \geq 1-\epsilon, \quad 		\nu_{2}^{(d)}( \frac{1}{C} \leq \frac{d \nu_1^{(d)}}{d \nu_2^{(d)}} \leq C)  \geq 1-\epsilon,
	\]
	and similarly for the measures $\nu^{(h)}$.
\end{thm}

\begin{proof}
	We first note that the statement is clearly stronger for the law on the height function than for the law on dimers since the former determines the latter locally.  Also as noted in \cref{rk:generalD}, proving \cref{prop:upper_bound_coupling} is not really necessary so let us focus on the proof of \cref{prop:lower_bound_coupling}.
	
	Compared to the initial case, we need to preserve the winding of UST branches between $D_1$ and $D_2$. Recall (as pointed out in \cref{rq:winding_modulo2pi}) that with our conventions, the value modulo $2\pi$ of the winding of a path only depends on its initial direction. In particular for any choice of the map $\phi$ from $D_2$ to $D_1$ and any face $f$ in $U$ we have $W( \gamma_f, \gamma_f(0) ) \equiv W( \phi(\gamma_f), \phi(\gamma_f(0))) [2\pi]$. Furthermore, it is not hard to see that if $\phi$ extends to a homeomorphism from $\overline{(D_2)}_d$ to $\overline{(D_1)}_d$ and if it sends the marked prime end of $d_2$ to the one of $D_1$, then one must have $W( \gamma_f, \gamma_f(0) ) - W( \phi(\gamma_f), \phi(\gamma_f(0))) = 2k \pi$ for some $k \in \Z$ which does not depend on $f$. Finally it is easy to see that composing if necessary with a map that twists $-k$ times the topological annulus $D_2 \setminus U$, we can always choose
	$\phi$ so that it preserves the winding of paths starting in $U$. Since $X_2$ is absolutely continuous with respect to a simple random walk, we can apply \cref{lem:harmonic_measure} to it. In particular we can without loss of generality assume that $ X_2(T_{\max})$ is at a macroscopic distance $r$ from the marked point of $\partial D_2$. 
	
	Now assume that in the construction of the walk $X_1$ (which is done exactly as in \cref{sec:RWD1}), we choose rectangles small enough compared to $r$. In the last step, by Beurling estimate, the diameter of the last added piece is very unlikely to be greater than $r/4$ and in particular with high probability the endpoint of $X_1$ must be within $d_{D_1}$-distance at most $r/2$ of $\phi( X_2(T_{\max}))$. Overall since the winding depends continuously on the path away from $U$, by taking $r$ small enough we can ensure that the windings of the UST branches in $D_1$ and $D_2$ are at most $\pi/2$ away from each other. On the other hand as pointed out in \cref{rq:winding_modulo2pi}, the value modulo $2\pi$ of the winding is determined independently of the tree, so the only possibility is that the windings of the UST branches in $D_1$ and $D_2$ agree. The rest of the proof follows as in \cref{sec:erasure,sec:multiplepath}.
	
\end{proof}


The continuity statement is slightly more involved because, as already pointed out in \cref{sec:domains} the classical topology on domains are not strong enough for us and we have to use the notion introduced in \cref{def:distance_domain}.

\begin{lem}\label{continuity_dimers}
	For all $U$, $V$ simply connected domains with $\overline U \subset V$ and all $\epsilon > 0$, there exists $r > 0$ such that for all $\delta$ small enough, for all simply connected domains $D_1, D_2$ with locally connected boundary and marked prime end, if $D_1$ and $D_2$ are $r$ close then
		 \[
	\nu_1^{(d)} ( 1-\epsilon \leq \frac{d \nu_1^{(d)}}{d\nu_2^{(d)}} \leq 1+\epsilon) \geq 1-\epsilon, \qquad \nu^{(d)}_2(  1-\epsilon \leq \frac{d \nu^{(d)}_1}{d\nu^{(d)}_2} \leq 1+\epsilon ) \geq 1-\epsilon,
	\]
	and similarly for $\nu^{(h)}$.
\end{lem}
\begin{proof}
	As in \cref{continuity_bounded}, it is enough to show that with high probability we can couple the laws of single branches in $D_1$ and $D_2$ so that they agree up to a neighbourhood of the boundary and create the same winding.
	
	Fix a face in $U^\d$ and let $X^1$, $X^2$, $\gamma_1$ and $\gamma_2$ denote the random walks and tree branches (going to infinity according to the convention from \cref{sec:winding}) in $D_1^\d$ and $D_2^\d$ respectively. Fix $\epsilon> 0$, $r$ to be chosen small enough later and let $f : D_1 \to D_2$ be an homomorphism given by \cref{def:distance_domain}. We couple $X^1$ and $X^2$ by setting $X^1= X^2$ up to $T := \inf \{ t : d(X^1[t], \partial (D_1 \cap D_2) \leq 2r \}$ and then letting them move independently. Reasoning as in \cref{continuity_bounded}, it is easy to see that if $r$ is small enough compared to $d (\partial U, \partial V)$ then $\gamma_1$ and $\gamma_2$ agree on $U$ with high probability. We also let $\tau_1$ and $\tau_2$ denote respectively the exit time of $D_1$ by $X^1$ and $D_2$ by $X^2$. Finally let $x_{\r}$ denote the marked prime end of $D_1$. 
	
	By \cref{beurling} and \cref{lem:harmonic_measure}, we know that there exists $\alpha > 0$ such that
	\[
	\P( X^1[\tau_1,T], X^1[\tau_1] ) \geq \sqrt{r}) \leq r^\alpha, \quad \P( d_{D_2}(X^2[T], X^2[\tau_2] ) \geq \sqrt{r}) \leq r^\alpha
	\]
and
\[
	\P( \abs{X^1[\tau_1] - x_{\r}} \leq 4\sqrt{r} ) \leq r^\alpha \quad \P( \abs{X^2[\tau_2] - f(x_{\r})} \leq 4\sqrt{r} ) \leq r^\alpha ,
\]
Putting the first two estimates together with the definition of $f$, we see that with high probability the distance between $X^2[0, \tau_2]$ and $f(X^1[\tau_1])$ is at most $3 \sqrt{r}$. Furthermore, we see from the last estimate that $X^2[\tau_2]$ and $f(X^1[\tau_1])$ cannot be on opposite sides of $f( x_{\r})$ so the extensions along $\partial D_1$ and $\partial D_2$ in the definition of $\gamma_1$ and $\gamma_2$ must also be close to each other. By continuity of the winding, $W( \gamma_1, \gamma_1(0))$ and $W( \gamma_2, \gamma_2(0))$ must therefore be close (still with high probability) and by \cref{rq:winding_modulo2pi} in that case they must be equal.
\end{proof}

For restriction of large domains, the law of the global shift in the height function depends on the whole domain so we only have a statement for $\nu^{(d)}$. Also in that case the boundary regularity is essentially irrelevant as long as there is some way to define a dimer configuration.
\begin{lem}
	For any bounded open set $U$ and $\epsilon > 0$, there exists $R > 0$ and $\delta_0 > 0$ such that for all $D_1, D_2$ which are either domains or the full plane, if $B(0, R) \subset D_1$ and $B(0, R) \subset D_2$ then 
	\[
	\nu^{(d)}_1 ( 1-\epsilon \leq \frac{d \nu^{(d)}_1}{d\nu^{(d)}_2} \leq 1+\epsilon) \geq 1-\epsilon, \qquad \nu^{(d)}_2(  1-\epsilon \leq \frac{d \nu^{(d)}_1}{d\nu^{(d)}_2} \leq 1+\epsilon ) \geq 1-\epsilon.
	\]
\end{lem}
\begin{proof}
Fix $\epsilon > 0$, it is well known that, for the full plane measure on trees, there exits $R >0$ such that, with probability at least $1-\epsilon$, all branches starting in $U$ merge within $B(0, R)$. In that case, the height difference between any two points in $U$ is given by the winding of a path staying inside $B(0, R)$ and in particular the dimer measure in $U$ is measurable with respect to the tree measure in $B(0,R)$. We conclude by applying \cref{continuity_unbounded} with $U$ replaced by $B(0,R)$.
\end{proof}

Another natural question is whether the global shift of the height function can be chosen (somewhat) independently of the dimer configuration, especially considering that the dimer configuration depends mostly on restriction of the tree in $U$ while the height depend on the number of turns of branches around $U$. The next two results give a positive answer to this question.
\begin{lem}\label{height_shift}
	Let $D$ be a simply connected bounded domain with a locally connected boundary and a marked boundary point and let $U$ be an open set such that $\overline U \subset D$. For all $\ell \in \Z$ and $\epsilon$ there exists $C> 0$ such that for all $\delta$ small enough,
	\[
	\nu^{(h)} \big( \frac{1}{C} \leq \frac{d \nu^{(h)}( . + \ell)}{d \nu^{(h)}} \leq C \big) \geq 1 - \epsilon,
	\]
	where by $\frac{d \nu^{(h)}( . + \ell)}{d \nu^{(h)}}$ we denote the Radon-Nikodym derivative of a height configuration shifted by $\ell$ with respect to the original measure.
\end{lem}
\begin{proof}
	This is essentially a corollary of the first part of the proof of \cref{thm:main_dimer} : in that proof we can take $D_1 = D_2$ but consider a map $\phi$ that adds exactly $2\pi \ell$ to the winding of any path from $U$ to $\partial D$ instead of a map that preserves the winding.
\end{proof}

\begin{corollary}\label{height_shift_conditionnal}
		Let $D$ be a simply connected bounded domain with a locally connected boundary and a marked boundary point and let $U$ be an open set such that $\overline U \subset D$. For all $\epsilon >0$, there exists $C > 0$ such that for all $\delta$ small enough, there exists a set $\mathcal{E}$ of dimer configurations of $U^{\d}$ such that
		\begin{itemize}
			\item $\nu^{(d)} (\mathcal{E}) \geq 1-\epsilon$,
			\item $\forall \omega \in \mathcal{E}, \nu^{(h)} \big( \forall \ell \in \Z, |\ell| \leq \frac{1}{\epsilon},  \frac{1}{C} \leq \frac{d \nu^{(h)}( . + \ell)}{d \nu^{(h)}} \leq C \big| \omega \big) \geq 1 - \epsilon$,
		\end{itemize}
	where in the second line, $\nu^{(h)} (\cdot | \omega)$ is the conditional law on the height function in $U$ given that the dimer configuration in $U$ is given by $\omega$.
\end{corollary}
\begin{proof}
	First by union bound, \cref{height_shift} implies that 
	\[\nu^{(h)} \big( \forall \ell \in \Z, |\ell| \leq \frac{1}{\sqrt{\epsilon}},  \frac{1}{C} \leq \frac{d \nu^{(h)}( . + \ell)}{d \nu^{(h)}} \leq C \big) \geq 1 - \sqrt{\epsilon}.
	\]
	  Let $G$ denote this event. We have $\E_{\nu^{(d)}} ( \nu^{h}( G^{c} | \omega)) \leq \sqrt{\epsilon}$ so by Markov inequality,  
	  \[\nu^{(d)} (  \nu^{(h)}( G^{c} | \omega) \geq \epsilon^{1/4}) \leq \epsilon^{1/4},
	  \] 
	  which concludes the proof taking $\mathcal{E} = \{ \omega : \nu^{(h)}( G| \omega) \geq 1- \epsilon^{1/4}\}$.
\end{proof}

\begin{remark}
	\cref{height_shift_conditionnal} essentially shows that it is possible to sample the dimer configuration in $U$ without revealing too much information over the global shift of the height in the sense that it cannot be concentrated in intervals smaller than $1/\epsilon$. A similar statement was obtained previously in \cite{Laslier2018} when $U$ is small enough compared to $D$. The method there follows more closely Wilson's algorithm in order to identify concrete random variables which are independent of the dimer configuration in $U$ but contribute to the global shift of height function through a notion of ``isolated scale''. This could certainly be adapted to the setting of this paper if one needed a more explicit construction of the set $\mathcal{E}$ or a more concrete link between $\epsilon$, the constant $C$ and the domain $D$.
\end{remark}

\subsection{Sampling an annulus}\label{sec:annulus}

In the previous sections, we always assumed $U$, $D_1$, $D_2$ to be simply connected. While relaxing this assumption for $D_1$ or $D_2$ introduces far to much complication for our purpose (see \cite{BLRtorus2} for information on dimers and Temperley's bijection in non simply-connected domains), it makes sense to consider a general open set $U$. For example, for possible applications to loop-ensembles it is very natural to ask whether one can reveal the existence of a loop in a specific annulus without affecting too much the distribution inside.

The first step is to prove that when running Wilson's algorithm, it does not ``cost'' to much to avoid a finite number of sets. More precisely, let $D$ be a simply connected domain and let $U_1', \ldots U_n'$ be disjoint open sets with $\overline{U_i'} \subset D$ and let $U_1, \ldots, U_n$ be simply connected open sets with $\overline{U_i} \subset U_i'$. Define $V = D \setminus \cup U'_i$. We will denote by $\cT^{ex}_V$  the union of all branches starting in $V$ (seen as a tree) and call its law $\nu^{ex}_V$ while we call $\cT_V$ and $\nu_V$ the restriction of the UST to $V$ and its law. 
\begin{lem}\label{annulus_coupling}
 For all $\epsilon > 0$, there exists $p > 0$ such that for all $\delta$ small enough there exists a law $\tilde \nu_V$ such that
	\[
		\tilde \nu^{ex}_V( \cT \cap (\cup U_i) = \emptyset ) \geq 1 - \epsilon , \quad  d_{TV} ( \tilde \nu_V , \nu_V ) \leq \epsilon, \quad \frac{d \tilde \nu^{ex}_V}{d \nu^{ex}_V} \leq \frac{1}{p}.
	\]
	Furthermore the total variation bound also holds for the restriction of the dimer height function and the dimer configuration.
\end{lem}
\begin{proof}
	Fix $u_1 \in U_1, \ldots, u_n \in U_n$. By Proposition 4.11 from \cite{BLR16}, a single loop-erased random walk started close to $V$ is unlikely to come very close to any of the $u_i$. Combined with Schramm's finiteness theorem, we see that we can find $r > 0$ such that $\nu^{ex}_V \cT \cap (\cup B(u_i, r)) = \emptyset ) \geq 1 - \epsilon$.
	
	Take $\phi$ an homomorphism from $D$ to itself which restricts to the identity over $D$ and such that $\phi( B(u_i, r) ) = U_i$. We can construct $\tilde \nu^{ex}_V$ as in \cref{sec:lower_bound}, we just need to replace the last condition in \cref{sec:RWD1} (about the behaviour when the walk comes close to $\partial D_2$) by the condition that the walk does not enter $\cup B(u_i, r)$. The proof for the spanning tree then follows exactly as above. 
	
	In this setting, the construction automatically preserves the winding of the branches (and hence the height function and dimer configuration) because the winding of a portion of path in some $U_i$ only depends on the points where it enters and exits $U_i$.
\end{proof}

Now in each of the $U_i$, we can apply the previous results comparing the law in the remaining region after sampling $\cT^{ex}_V$ with the original domain to get the following.
\begin{prop}\label{prop:U_multiple}
	Let $\nu_{V,U_i}$ (resp. $\nu_{V,U_i}^{(h)}$, $\nu_{V,U_i}^{(d)}$) be the law of the restriction of the tree (resp. the height function and dimer configuration) to $V \cup \cup_i U_i$ and let $\nu_V$, $\nu_{U_i}$ be the corresponding marginal distributions. For all $\epsilon > 0$ there exists $C >0$ such that for all $\delta$ small enough
	\begin{gather*}
	\nu_{V, U_i} \big(\frac{1}{C} \leq \frac{d \nu_{V, U_i}}{d \nu_V \otimes \nu_{U_1} \otimes \ldots \otimes \nu_{U_n}} \leq C  \big) \geq 1-\epsilon, \\
		\nu_V \otimes \nu_{U_1} \otimes \ldots \otimes \nu_{U_n} \big(\frac{1}{C} \leq \frac{d \nu_{V, U_i}}{d (\nu_V \otimes \nu_{U_1} \otimes \ldots \otimes \nu_{U_n})} \leq C  \big) \geq 1-\epsilon.
	\end{gather*}
	The same holds for the dimer and height function measures.
\end{prop}
In other words, the true joint law of the trace of the UST (or dimer/height function) is comparable with an independent product of the marginal distribution, in the Radon-Nikodym sense.
\begin{proof}
	We start by the equation in the first line of the display, where the lower bound is the non-trivial part. Also we focus first on the laws on trees. 
	
	Let $\tilde \nu$ be defined as follow. First we sample $\cT_V^{ex}$ as in \cref{annulus_coupling}, with the additional condition that $\cT_V^{ex}$ does not come within distance $r$ of $\cup U_i$ for some $r > 0$, which only amounts to changing the definition of the $U_i$. Then in each $U_i$, we sample independently the tree with boundary conditions given by $\cT_V^{ex}$, using in each case the law from \cref{prop:lower_bound_coupling} to compare it with the law of $\nu_{U_i}$.
	
	By construction $\tilde \nu$ satisfies $d_{TV}( \tilde \nu,  \nu_V \otimes \nu_1 \otimes \ldots \otimes \nu_{U_n} ) \leq \epsilon$. Furthermore by the continuity result \cref{continuity_bounded}, the constants appearing in each use of \cref{prop:lower_bound_coupling} can be chosen uniformly over all realizations of $\cT^{ex}_V$, therefore $\frac{d \tilde \nu}{d \nu_{V, U_i}} \leq 1/p$ for some $p$ since this is true for the marginal on $\cT_V^{ex}$ and then true for the conditional laws of all restriction to the $U_i$ given $\cT_{V}^{ex}$. Given this coupling, we conclude as in the proof of \cref{prop:lower_bound}. The proof of the second equation is similar.
	
	Finally for the laws on dimers or height function, the proof is very similar except that the weaker continuity result from \cref{continuity_dimers} has to be used instead of \cref{continuity_bounded}. More precisely fix two intermediate open set $V', V''$ such that $\overline V \subset V'$, $\overline{V'} \subset V''$ but $\overline V'' \cap \overline U_i = \emptyset$ for all $i$. By the finiteness theorem \cref{thm:finitude}, choose $k$ points in $V''$ such that, given the subtree $\cT_k$ emanating from these points, except on a event of probability $\epsilon$ all random walks and all branches started in $V$ will stay in $\partial V''$ and no branch started from $D \setminus V''$ can enter $V'$. It is well known (see for example Section 3.4 in \cite{Lawler2004}) that the laws of loop-erased random walk are tight (as $\delta \to 0$) when viewed as simple curves for the uniform topology on curve up to reparametrisation. Therefore, using the fact that only $k$ branches have been sampled, by \cref{continuity_dimers,slit_continuity}, for all $\epsilon$ we can find $C$ such that with high probability \cref{thm:main_dimer} can be used with the constant $C$ to compare the $\nu_{U_i}( \cdot | \cT_k)$ to the $\nu_{u_i}$. This concludes because, by construction, the restrictions to the $U_i$ and $\cT_V$ are almost independent given $\cT_k$.
\end{proof}



\subsection{Continuum limit}

In this section, we will show the analogue of \cref{thm:main} for the continuum spanning tree.

\begin{prop}
	Let $D_1$, $D_2$ be two simply connected domains and let $U$ be an open set such that $\overline{U} \subset D_1 \cap D_2$. Let $\nu_1^c$ and $\nu_2^c$ be the laws of the restriction to $U$ of the continuum spanning tree in $D_1$ and $D_2$. These laws are absolutely continuous with respect to each other.
\end{prop}
Before giving the proof which will follow from \cref{thm:main}, let us note that this results should be fairly easy to prove using the theory of imaginary geometry. Indeed the restriction of the trees to $U$ are measurable with respect to the restriction of the associated GFFs to $U$ but since the GFF is a Gaussian process, it should be possible to write the Radon-Nicodym
derivative explicitly. However we still think that the following proof is useful because it does not rely on any knowledge of SLE or the GFF. For example it sounds much easier to adapt this proof to a non simply-connected setting than to develop imaginary geometry on Riemann surfaces.


\begin{proof}
	In this proof, consider $G^\d = \delta \Z^2$ with unit weight. By the celebrated result of \cite{Lawler2004}, the laws $\mu_i^\d$ of the discrete uniform spanning tree in $\D_i^\d$ converge in law to continuum limits $\mu_i^c$ say for Schramm's topology which defines a Polish space.
	
	Fix $\epsilon > 0$ and write 
	\[
	\bar\nu_i^\d = \nu_i^\d 1_{\frac{1}{C} \leq \frac{d \nu_1^\d}{ d \nu_2^\d} \leq C},
	\]
	 where $C$ is obtained from \cref{thm:main} with $i = 1,2$. Note that, since $\bar \nu_i^\d \leq \nu_i^\d$ and the sequence $\nu_i^\d$ converge in law, the sequence $\bar \nu_i^\d$ must be tight so let $\bar \nu_i^c$ be a sub-sequential limits. By the choice of $C$ in \cref{thm:main}, both $\bar \nu_1^c$ and $\bar \nu_2^c$ have total mass at least $1-\epsilon$.
	
	By definition, for any positive continuous bounded function $f$,  $\bar \nu_1^\d(f) \leq \nu_1^\d(f)$ so taking the limit along the appropriate subsequence $\bar \nu_1 (f)^c \leq \nu_1^c(f)$ and therefore $\nu_1^c - \bar \nu_1^c$ is a positive bounded linear form on continuous functions. By the Riesz representation theorem, it extends to a unique positive measure but since $\nu_1^c - \bar \nu_1^c$ is already a measure we must have $\bar \nu_1 (E)^c \leq \nu_1^c(E)$ for any measurable event $E$. Of course the analogous statement holds for $\bar \nu^c_2$ and $\nu^c_2$. Now note that the same proof and the definition of the $\bar \nu_i^c$ also shows that
	\[
	\frac{1}{C} \bar \nu_2^c(E) \leq \bar\nu^c_1(E) \leq C \bar \nu_2^c( E).
	\]
	Combining the above with the bound on the total mass of the $\bar\nu_i$, we obtain that
	\[
		\frac{1}{C} ( \nu_2^c(E) - \epsilon) \leq \nu^c_1(E) \leq C \ \nu_2^c( E) + \epsilon
	\]
	for any event $E$ and this concludes the proof since $\epsilon$ was arbitrary.
\end{proof}

\printbibliography

\end{document}